\documentclass[twoside,11pt,reqno]{amsart}
\usepackage[a4paper,margin=1.2in]{geometry}
\usepackage[dvipsnames]{xcolor}
\usepackage{amsthm,amsmath,amsfonts,graphicx,geometry,lipsum,amsxtra}
\usepackage{hyperref}
\usepackage{mathrsfs}
\usepackage{verbatim}
\usepackage{appendix}
\usepackage{tikz}
\usepackage{epigraph}
\usepackage{tcolorbox}

\numberwithin{equation}{section}
\usepackage{esint}
\usepackage{mathtools}

\usepackage[dvipsnames]{xcolor}
\usepackage{etoolbox}
\hypersetup{
    colorlinks=true,
    linkcolor=blue,
    citecolor=blue,
    urlcolor=red,
    pdftitle={[BMSS25]Resonance near a degenerate embedded eigenvalue},
    }
\usepackage{comment}
\usepackage{amsmath}
\usepackage{tikz}
\usepackage{epigraph}
\usepackage{lipsum}
\allowdisplaybreaks

\definecolor{titlepagecolor}{cmyk}{1,.60,0,.40}
\patchcmd{\subsection}{\normalfont}{\normalfont\color{blue}}{}{}
\DeclareFixedFont{\titlefont}{T1}{ppl}{b}{it}{0.5in}

\def\th@plain{%
  \thm@notefont{}
  \itshape 
}
\def\th@definition{%
  \thm@notefont{}
  \normalfont 
}
\makeatother

\usepackage{srcltx} 
\usepackage{amscd}
\usepackage{amssymb}
\usepackage{amsmath}
\usepackage{pb-diagram}
\usepackage{graphicx}
\usepackage{hyperref}
\usepackage{array}
\usepackage[all]{xy}
\xyoption{matrix}
\xyoption{arrow}
\usepackage{pdfsync}
\usepackage{physics}
\usepackage{enumitem}
 
\usepackage{thmtools}
\usepackage{amssymb}

\DeclareUnicodeCharacter{2212}{-,+}
 \numberwithin{equation}{section}
\newtheorem{theorem}{Theorem}[section]
\newtheorem{proposition}[theorem]{Proposition}
\newtheorem{corollary}[theorem]{Corollary}
\newtheorem{lemma}[theorem]{Lemma}
\newtheorem{remark}[theorem]{Remark}
\newtheorem{definition}[theorem]{Definition}

\newcommand{\bdfn}{\begin{definition}}
\newcommand{\bthm}{\begin{theorem}}
\newcommand{\blem }{\begin{lemma}}
\newcommand{\bcla }{\begin{cla}aim}
\newcommand{\ecla }{\end{cla}im}
\newcommand{\bpro}{\begin{proposition}}
\newcommand{\bcor }{\begin{corollary}}

\newcommand{\brmrk}{\begin{remark}}
\newcommand{\bassu}{\begin{assumption}}
\newcommand{\eassu}{\end{assumption}}
\newcommand{\edfn}{\end{definition}}
\newcommand{\ethm}{\end{theorem}}
\newcommand{\elem }{\end{lemma}}
\newcommand{\epro}{\end{proposition}}
\newcommand{\ecor}{\end{corollary}}
\newcommand{\ermrk}{\end{remark}}

\newcommand*\diff{\mathop{}\!\mathrm{d}}
\newcommand{\e}{\mathrm{e}}
\newcommand{\CC}{\mathbb C}
\newcommand{\RR}{\mathbb R}
\renewcommand{\H}{\mathcal H}

\renewcommand{\iota}{{i\mkern1mu}}
\renewcommand{\Re}{\operatorname{Re}}
\renewcommand{\Im}{\operatorname{Im}}

\newcommand{\slim}{\mathop{\mathrm{s\text{-}lim}}}
\calclayout
\begin{document}
\title[Resonance near a doubly degenerate embedded eigenvalue]{Resonance near a doubly degenerate embedded eigenvalue}

\author[Bansal]{Hemant Bansal}
\address{Hemant Bansal\\ Department of Mathematical Sciences, Indian Institute of Science Education and Research Mohali\\Sector 81, SAS Nagar, Punjab 140306, India}
\email{ph20030@iisermohali.ac.in}

\author[Maharana]{Alok Maharana}
\address{Alok Maharana\\ Department of Mathematical Sciences, Indian Institute of Science Education and Research Mohali\\Sector 81, SAS Nagar, Punjab 140306, India}
\email{maharana@iisermohali.ac.in}

\author[Sahu]{Lingaraj Sahu}
\address{Lingaraj Sahu\\ Department of Mathematical Sciences, Indian Institute of Science Education and Research Mohali\\Sector 81, SAS Nagar, Punjab 140306, India}
\email{lingaraj@iisermohali.ac.in}

\vspace*{-1cm}

 \maketitle
\begin{abstract}

This paper extends the study of resonance phenomenon initiated by the authors in~\cite{LS}
to the case of doubly degenerate embedded eigenvalues (i.e. eigenvalue of multiplicity two). A fundamentally new concept is  introduced to resolve the difficulties that arise in this study, beyond the methods of \cite{LS}. We apply a differential topological technique, namely the Morse Lemma, to study the present case. This allows us to understand rank-two self-adjoint perturbations of the Laplacian on $L^{2}(\mathbb{R}^{3})$, and along with methods of \cite{LS}, we obtain asymptotic results for the spectral density near a doubly degenerate embedded eigenvalue.
Importantly, we are able to easily handle the threshold eigenvalue case. 

\par We also analyze important properties which explain such resonance phenomenon, viz., asymptotic behaviour of the sojourn time, scattering cross-section and time delay.
\end{abstract}

\section{Introduction}
Under perturbations, embedded eigenvalues of self-adjoint operators may
disappear into the continuous spectrum and give rise to resonances, which can be
described through the local behaviour of quantities such as the spectral
density, scattering cross-section, time delay and sojourn time. The resonance phenomenon associated with such eigenvalues has been studied by various authors; see, for example,
\cite{howland,simon,Orth,Jensen,Astaburuaga2024}.

In a recent work~\cite{LS}, a detailed analysis has been carried out for rank-one
perturbations of the Laplacian on $L^2(\RR^3)$, where the disappearance of a simple embedded eigenvalue is shown
to produce a Breit--Wigner-type asymptotic behaviour for the spectral density, scattering cross-section and
time delay near the resonance. In this paper, we consider one-parameter perturbations
of the Laplacian on \(L^{2}(\mathbb{R}^{3})\) by suitable rank-two operators and study the resonance phenomenon due to a doubly degenerate embedded eigenvalue (i.e. eigenvalue of multiplicity two). Methods of \cite{LS} (use of the implicit function theorem) are not enough to resolve the degeneracy that arise in the study. We use a tool from differential topology, namely the Morse Lemma, to obtain two distinct resonance paths in this case. We associate a normalized eigenvector to each of the resonance paths in a canonical way so that they together form an orthonormal eigenbasis for the embedded eigenvalue and show that the spectral density associated with each of them admits a
Breit--Wigner asymptotic profile near the corresponding resonance path. As a consequence, we prove the spectral concentration by analyzing the spectral projection along the two paths separately. Finally, we carry out asymptotic analysis of the scattering amplitude, time delay and sojourn time. 

\par The higher degeneracy case presents some new difficulties which is under investigation and will be reported in a future work. 

\par
The paper is organized as follows. Section~\ref{2} collects preliminary results
and fixes the notation used throughout the paper. In Section~\ref{3}, we recall
the spectral representation of the Laplacian and establish some related results. Section~\ref{4} develops the
spectral theory of rank-two perturbations of the Laplacian in
$L^{2}(\mathbb{R}^{3})$, including necessary and sufficient conditions for the
existence of embedded eigenvalues of multiplicity two. In Section~\ref{5}, we specify the class of rank-two perturbations and 
fix the model that will be studied in the remainder of the paper.
Section~\ref{6} is devoted to the asymptotic analysis of the spectral density. In Section~\ref{7}, we prove the spectral concentration, time decay  and derive the estimates for the sojourn time. Finally, Section~\ref{8}
examines the asymptotic behaviour of the scattering amplitude as the perturbed
scattering system approaches the unperturbed one and derives corresponding
asymptotic results for time delay.

\section{Notations and some basic results}\label{2}
For a self-adjoint operator $T$ (possibly unbounded) on a separable Hilbert space $\H$, let $\sigma(T)$ and $\sigma_{ac}(T)$ denote the spectrum and the absolutely continuous spectrum of $T$ respectively. Let $\mathcal{H}_{\text{ac}}(T)$ denote the absolutely continuous subspace with respect to $T$. Let $\mathcal B(\H)$ denotes the space of bounded operators on $\H$. A sequence $\{T_n\}\subset\mathcal B(\H)$ is said to converge strongly to an operator $T\in\mathcal B(\H)$, written as $\slim\limits_{n\to\infty}~T_n=T$, if for each $v\in\H$, $T_nv\to Tv$ as $n\to\infty.$\par Let $C^\infty(\mathbb{R}^n)$ and $\mathcal{S}(\mathbb{R}^n)$ denote the spaces of all smooth functions and Schwartz class functions on $\mathbb{R}^n$ respectively. For $1 \leq p\leq \infty$, let $L^p(\mathbb{R}^n)$ denote the standard Lebesgue space. The Fourier transform of a function $f:\RR^n\to\CC$ is defined as:
\[\widehat{f}(\xi) = \frac{1}{(2\pi)^{\frac{n}{2}}} \int_{\mathbb{R}} f(x) e^{-\iota \xi\cdot x}\diff x \quad \text{for}\ \xi \in\RR^n\]
whenever the integral converges in an appropriate sense.\\
We now present some preliminary results that will be required for computing the resolvent limits of the Laplacian.
\bdfn\label{CPV-dfn} Let $f \in \mathcal{S}(\mathbb{R})$. The Cauchy principal value of $f$, denoted by $\gamma(f,\cdot)$, is defined as
\begin{equation}\label{pvdfn}\gamma(f,\lambda): = \lim_{\epsilon \to 0^+} \int_{\mathbb{R} \setminus (\lambda - \epsilon, \lambda + \epsilon)} \frac{f(x)}{x - \lambda}\diff x\quad \text{ for }\lambda\in\RR\end{equation} where the above limit exists~$($see \cite[Corollary~$4.95$~$($b$)$]{DorinaMitreaBook}$)$. \edfn
\begin{proposition}[{\cite[Proposition~2.2]{LS}}]\label{CPV-properties}
Let $f \in \mathcal{S}(\mathbb{R})$. Then $\gamma(f,\cdot)\in C^\infty(\mathbb{R})$ and its $k$-th derivative satisfies
    \begin{equation}\label{derivativeprincipalvalue}
        \frac{\partial^k\gamma(f,\cdot)}{\partial\lambda^k}
        = \gamma\!\left(\frac{\diff^k f}{\diff \lambda^k},\cdot\right).
    \end{equation}
\end{proposition}
\begin{proposition}[Plemelj--Privalov Theorem, {\cite[eq.~(4.7.48)]{DorinaMitreaBook}}]\label{Plemelj-Privalov Theorem}
Let $f \in \mathcal{S}(\mathbb{R})$. Then
\[
    \lim_{\epsilon \to 0^+} 
    \int_{\mathbb{R}} \frac{f(x)}{x - (\lambda \pm \iota \epsilon)}\,\diff x
    = \gamma(f,\lambda) \pm \iota \pi f(\lambda),
\]
uniformly for $\lambda$ in compact sets.
\end{proposition} \section{Spectral representation of the Laplacian on $L^2(\RR^3$)}\label{3}In this section we present several preliminary results concerning the Laplacian 
\(\Delta = \sum_{j=1}^{3} \frac{\partial^{2}}{\partial x_{j}^{2}}\) in \(\mathbb{R}^{3}\), 
its spectral representation and the boundary behaviour of its resolvent.\par
Let $H_0$ denote the self-adjoint extension of $-\Delta$ on the Hilbert space $L^2(\RR^3)$, with domain\[\mathcal D(H_0) =\left\{v\in L^2(\RR^3):\int_{\RR^3}\left||\xi|^2\hat v(\xi)\right|^2\diff \xi<\infty\right\}.\] The operator $H_0$ is purely absolutely continuous with $\sigma(H_0)=[0,\infty)$. 
The spectral representation of $H_0$ is realized in the unitarily isomorphic Hilbert space $L^2([0,\infty), L^2(S^2))$ via the unitary map $U$ defined by\begin{equation}\label{iso}
    (Uv)_\lambda(\omega)
    = 2^{-1/2}\lambda^{1/4}\widehat{v}(\sqrt{\lambda}\,\omega)
    \quad \text{for almost every } \lambda \in [0,\infty),\, \omega \in S^2,
\end{equation}
where $S^2$ is the unit sphere in $\mathbb{R}^3$ equipped with the surface area measure and $[0,\infty)$ is endowed with the Lebesgue measure. Note that for all $v\in \mathcal D(H_0) $, we have \[(UH_0v)_\lambda=\lambda(Uv)_\lambda \quad \text{for almost every } \lambda \in [0,\infty).\]For notational simplicity, we denote $(Uv)_\lambda$ by $v_\lambda$ and the $n$-th strong derivative 
$\frac{\diff^n}{\diff\lambda^n}v_\lambda$ at $\lambda = s$ by $v_s^{(n)}$. We say that $v_\lambda$ \emph{vanishes to order $n$ at} $\lambda = s$ if 
$v_s^{(k)} = 0$ for $0 \le k < n$ and $v_s^{(n)} \neq 0$.
Define the space
\[
\mathcal{E}
:=
\left\{
v \in L^{2}(\mathbb{R}^{3})
\;\middle|\;
\begin{array}{l}
\text{the map } \lambda \mapsto v_\lambda \text{ is strongly smooth,} \\[0.2em]
\text{and decays rapidly as } \lambda \to 0
\text{ and } \lambda \to \infty
\end{array}
\right\}.
\]
This is a dense subspace of $L^{2}(\mathbb{R}^{3})$, since its image under
the spectral transform $U$ contains
$C_c^\infty\bigl((0,\infty); L^{2}(S^{2})\bigr)$, which is dense in
$L^{2}\bigl([0,\infty); L^{2}(S^{2})\bigr)$.\\Let $R_{0}(z)$ denote the resolvent $(H_0-z)^{-1}$ of $H_{0}$,
defined for $z \in \mathbb{C}\setminus[0,\infty)$. For $v,w\in\mathcal{E},$ define \[\eta_{v,w}(\lambda) = \begin{cases}\langle v_\lambda, w_{\lambda} \rangle &\text{ if }\lambda\geq0,\\0&\text{ if }\lambda<0.\end{cases}\]Note that, by Cauchy-Schwarz inequality, $\eta_{v,w}\in\mathcal S(\RR).$ Then for $\Im z\neq 0,$\[ \langle R_0(z)v, w \rangle=\int_0^\infty\langle( R_0(z)v)_\lambda,w_\lambda\rangle\diff\lambda=\int_\RR\frac{\eta_{v,w}(\lambda)}{\lambda-z}\diff\lambda.
\]  Hence by Proposition~\ref{Plemelj-Privalov Theorem}, we have
\begin{equation}\label{eqn-laplacianresolventlimit}
    \lim_{\epsilon \to 0^+} 
    \langle R_0(\lambda \pm \iota\epsilon)v, w \rangle
    = \gamma(\eta_{v,w}, \lambda)
      \pm \iota\pi\, \eta_{v,w}(\lambda)
\end{equation}uniformly for $\lambda$ in compact sets. \\For $v \in \mathcal{E}$ satisfying   
$v_a=0$ in $L^2(S^2)$ for some $a\geq0$, we define the function $\Phi_{v,a}\in L^2(\RR^3)$ by
\begin{equation}\label{modified}
(\Phi_{v,a})_\lambda =
\begin{cases}
\dfrac{v_\lambda}{\lambda-a}, & \text{if } \lambda \neq a, \\[6pt]
v'_a, & \text{if } \lambda = a.
\end{cases}
\end{equation}
Note that, when $a=0$, we have $v'_a = 0$. The following lemma will be useful in Section \ref{4}.
\begin{lemma}\label{derivative}
Let $v \in \mathcal{E}$ be such that
$v_{a}=0$ for some $a\geq0$. Then the following statements hold:
\begin{enumerate}[label=(\alph*)]
\item
$\Phi_{v,a} \in \mathcal{E}$.

\item
\[
\lim_{\epsilon \to 0^{+}} R_{0}(a+\iota\epsilon)v = \Phi_{v,a}
\quad \text{in } L^{2}(\mathbb{R}^{3}).
\]

\item
Moreover, if $w \in \mathcal{E}$ satisfies
$w_{a}=0$, then
\[
\left.\frac{\partial}{\partial \lambda}
\gamma(\eta_{v,w},\lambda)\right|_{\lambda=a}
=
\langle \Phi_{v,a},\, \Phi_{w,a} \rangle.
\]
\end{enumerate}
\end{lemma}

\begin{proof}
Since $v \in \mathcal{E}$, the function
$\lambda \mapsto (\Phi_{v,a})_\lambda$ is strongly smooth for $\lambda \neq a$
and decays rapidly as $\lambda \to 0$ and $\lambda \to \infty$.
Moreover, since $v_a=0$, repeated application of Taylor's theorem at $\lambda=a$ shows that the function
$\lambda\mapsto(\Phi_{v,a})_\lambda$ is strongly differentiable to arbitrary order at $\lambda=a$ and hence $\Phi_{v,a} \in \mathcal E.$ This proves part~(a).

To prove part~(b), observe that
\[
\lim_{\epsilon\to0^+}\frac{v_\lambda}{\lambda-(a+\iota\epsilon)}
=
(\Phi_{v,a})_\lambda
\quad \text{in } L^{2}(S^{2})
\quad \text{for } \lambda \neq a,
\]
and that
\[
\left\|
\frac{v_\lambda}{\lambda-(a+\iota\epsilon)}
-
(\Phi_{v,a})_\lambda
\right\|_{L^{2}(S^{2})}
\le
2\,\|(\Phi_{v,a})_\lambda\|_{L^{2}(S^{2})}.
\]
Furthermore,
\[
\|R_{0}(a+\iota\epsilon)v - \Phi_{v,a}\|^{2}
=
\int_{0}^{\infty}
\left\|
\frac{v_\lambda}{\lambda-(a+\iota\epsilon)}
-
(\Phi_{v,a})_\lambda
\right\|_{L^{2}(S^{2})}^{2}
\, \mathrm d\lambda.
\]
Since $\Phi_{v,a}\in \mathcal{E}$, the integrand is dominated by an
$L^{1}$-function and part (b) therefore follows from the Lebesgue dominated
convergence theorem. For $v,w\in\mathcal E$ with $v_a=w_a=0$ \[\eta_{v,w}(a)=\langle v_a,w_a\rangle=0\quad\text{and}\quad\eta'_{v,w}(a)=\langle v'_a,w_a\rangle+\langle v_a,w'_a\rangle=0.\] Hence the function $\lambda \longmapsto \frac{\eta'_{v,w}(\lambda)}{\lambda-a}
$
belongs to $L^{1}(\mathbb{R})$. By~\eqref{derivativeprincipalvalue}, we obtain
\[
\left.\frac{\partial}{\partial \lambda}
\gamma(\eta_{v,w},\lambda)\right|_{\lambda=a}
=
\gamma(\eta'_{v,w},a)
=
\int_{\mathbb{R}}
\frac{\eta'_{v,w}(\lambda)}{\lambda-a}
\, \mathrm d\lambda.
\]
An integration by parts yields
\[
\int_{\mathbb{R}}
\frac{\eta'_{v,w}(\lambda)}{\lambda-a}
\, \mathrm d\lambda
=
\int_{\mathbb{R}}
\frac{\eta_{v,w}(\lambda)}{(\lambda-a)^{2}}
\, \mathrm d\lambda
=\int_0^\infty\left\langle \frac{v_\lambda}{\lambda-a},\frac{w_\lambda}{\lambda-a}\right\rangle\diff\lambda=
\langle \Phi_{v,a},\, \Phi_{w,a} \rangle,
\]
which proves part~(c).
\end{proof}
\section{Rank-two perturbation of the Laplacian on $L^2(\RR^3)$}\label{4} 
\noindent Let \[V=\sum_{j=1}^2\langle\cdot,u_j\rangle u_j\] where $u_1,u_2\in \mathcal{E}$ and $u_1\perp u_2$. For $\alpha\in\mathbb R,$ consider the perturbed operator \[H_\alpha:=H_0+\alpha V\quad\text{on }L^2(\RR^3).\] Since $V$ is a finite-rank self-adjoint operator, the operator $H_\alpha$ is self-adjoint with domain $\mathcal{D}(H_\alpha)=\mathcal{D}(H_0)$, and its essential spectrum is $[0,\infty)$. We denote by $R_{\alpha}(z)$, the resolvent $(H_{\alpha}-z)^{-1}$ of $H_{\alpha}$,
defined for $\Im z \neq 0$ and by $E_{\alpha}$ its associated spectral measure. For notational simplicity, we denote $(u_j)_\lambda$ by $u_{j,\lambda}$ and $\eta_{u_j,u_k}$ by $\eta_{jk}$ for $j,k\in\{1,2\}$. 

We recall Stone's formula, which relates the spectral projections of $H_\alpha$ to the boundary values of its resolvent. For any $v,w \in L^2(\RR^3)$,
\begin{equation}\label{stone}
\frac{1}{2}\left(\langle E_\alpha(a,b)v,w\rangle+\langle E_\alpha[a,b]v,w\rangle\right)
= \lim_{\varepsilon\to 0^+}\,\frac{1}{2\pi i}\int_a^b
\big\langle \big( R_\alpha(\lambda+i\varepsilon)-R_\alpha(\lambda-i\varepsilon)\big)v,w \big\rangle\, d\lambda.
\end{equation} 
Thus, the analysis of spectral properties of $H_\alpha$ reduces to understanding the boundary behaviour of its resolvent as the parameter $\lambda\pm\iota\epsilon$ approaches the spectrum. To this end, we first relate the resolvent of $H_\alpha$ to that of $H_0.$

Let us define the linear map $\tau :L^2(\RR^3)\to\CC^2$ as follows: \begin{equation}
    \label{tau}\tau v= \begin{bmatrix}
          \langle v,u_1\rangle \\
\langle v, u_2\rangle \\
         \end{bmatrix}.\end{equation}
 Then $\tau ^*\tau =V$ and for $\Im z \neq 0$, the resolvent identity yields
\[
    R_\alpha(z) - R_0(z)
    = -\alpha R_0(z) V R_\alpha(z)
    = -\alpha R_0(z)\tau^* \tau R_\alpha(z).
\]
This implies \[
    B(\alpha,z)\tau R_\alpha(z)
    = \tau R_0(z)
\] where $B(\alpha,z)=I + \alpha \tau R_0(z)\tau^*.$ Note that the operator $B(\alpha,z)$ is invertible for $\Im z \neq 0$ and $B(\alpha,z)^{-1}=I - \alpha\,\tau R_{\alpha}(z)\tau^{*}.$ Consequently, we obtain
\begin{equation}\label{b(alpha)}
\tau R_{\alpha}(z)
=
B(\alpha,z)^{-1}\,\tau R_{0}(z).
\end{equation}   

From now on, we regard \( B(\alpha, z) \) as a \( 2 \times 2 \) matrix with respect to the standard orthonormal basis \( \{e_1, e_2\} \) of \( \mathbb{C}^2 \).
The following lemma relates the matrix elements of the resolvents of $H_\alpha$ and $H_0$.
\begin{lemma}\label{newlemma}
    For $v,w\in L^2(\RR^3)$ and $\Im z\neq0$, we have 
    \begin{equation}\label{R1}
        \langle  R_\alpha(z)v,w\rangle =\langle R_0(z)v,w\rangle-\frac{\alpha}{\det\,B(\alpha,z)}\sum_{j,k=1}^N\tilde{b}_{jk}(\alpha,z)\langle R_0(z)v,u_j\rangle \langle R_0(z)u_k,w\rangle
    \end{equation}where $\tilde{b}_{jk}(\alpha,z)$ denote the $(j,k)$th entry of the cofactor matrix associated with
$B(\alpha, z).$
\end{lemma}
\begin{proof}
    By the resolvent identity, we obtain  \begin{equation}\label{R2}\langle R_\alpha(z)v,w\rangle=\langle R_0(z)v,w\rangle-\alpha\sum_{j=1}^N\langle R_0(z)v,u_j\rangle \langle R_\alpha(z)u_j,w\rangle.\end{equation}
To compute $\langle R_\alpha(z)u_j,w\rangle$, by \eqref{b(alpha)} and the fact that $B(\alpha,z)^*=B(\alpha,\overline{z})$, we have\[\langle  R_\alpha(z)w,u_j\rangle=\langle \tau R_\alpha(z)w,e_j\rangle=\langle R_0(z)w,\tau^*B(\alpha,\overline{z})^{-1}e_j\rangle.\]
Now,
\begin{align*}\tau^*B(\alpha,\overline{z})^{-1}e_j&=\frac{1}{\det\,B(\alpha,\overline{z})}\tau^*(\tilde{b}_{j1}(\alpha,\overline{z}),\dots,\tilde{b}_{jN}(\alpha,\overline{z}))^{\top}=\frac{1}{\det\,B(\alpha,\overline{z})}\sum_{k=1}^N\tilde{b}_{jk}(\alpha,\overline{z})u_k.\end{align*}
       Therefore,  \begin{align*}\lim_{\epsilon\to0^+}\langle  R_\alpha(z)w,u_j\rangle&=\frac{1}{\det\,B(\alpha,z)}\sum_{k=1}^N\overline{\tilde{b}_{jk}(\alpha,\overline{z})}\langle R_0(\lambda\pm\iota0)w,u_k\rangle\end{align*}which implies \begin{equation} \label{R3}
         \langle  R_\alpha(z)u_j,w\rangle=  \overline{ \langle  R_\alpha(\overline{z})w,u_j\rangle}=\frac{1}{\det\,B(\alpha,z)}\sum_{k=1}^N\tilde{b}_{jk}(\alpha,z)\langle R_0(z)u_k,w\rangle.
       \end{equation} 
       Substituting \eqref{R3} into \eqref{R2}, \eqref{R1} follows.
\end{proof}
Note that, since $u_1,u_2 \in \mathcal{E}$,
it follows from~\eqref{eqn-laplacianresolventlimit} that
 \begin{equation}\label{uiresolventlimit}
      \lim_{\epsilon \to 0^+} 
    \langle R_0(\lambda \pm \iota\epsilon)u_j, u_k \rangle
    = \gamma(\eta_{jk}, \lambda)
      \pm\iota\pi\, \eta_{jk}(\lambda),\qquad j,k\in\{1,2\}
\end{equation}uniformly for $\lambda$ in compact sets. Hence, for $\lambda\in\RR$, we define $B(\alpha,\lambda\pm\iota0)\in\mathcal{B}(\CC^2)$ by \[B(\alpha,\lambda\pm\iota0) := \lim_{\epsilon \to 0^+} B(\alpha,\lambda \pm \iota\epsilon) \]where the limit is taken in the strong operator topology on
$\mathcal{B}(\mathbb{C}^{2})$.  
 
Here the strong limit is equivalent to
the convergence of all matrix elements with respect to the standard basis, which
is guaranteed by~\eqref{uiresolventlimit} and $B(\alpha,\lambda+\iota0)$ admits the
decomposition
\[B(\alpha,\lambda+\iota0)
    = C(\alpha, \lambda) + \iota D(\alpha, \lambda),\] where $C(\alpha,\lambda)$ and $D(\alpha,\lambda)$ are Hermitian matrices with entries  
\begin{equation}
 c_{jk}(\alpha, \lambda)
    = \delta_{jk} + \alpha\, \gamma(\eta_{kj}, \lambda),\quad d_{jk}(\alpha, \lambda)= \alpha\pi\, \eta_{kj}(\lambda).
\end{equation} Since $B(\alpha,z)^*=B(\alpha,\overline{z})$ for $\Im z\neq0$, it follows that \[B(\alpha,\lambda-\iota0)=B(\alpha,\lambda+\iota0)^*=C(\alpha,\lambda)-\iota D(\alpha,\lambda).\]
Note that, if $v,w\in \mathcal E$, then by Lemma \ref{newlemma} and \eqref{eqn-laplacianresolventlimit}, \[\langle R_\alpha(\lambda\pm\iota0)v,w\rangle:=\lim\limits_{\epsilon\to0^+}\langle R_\alpha(\lambda\pm\iota\epsilon)v,w\rangle\]exists if and only if $B(\alpha,\lambda\pm\iota0)$ is invertible. Let $\tilde{b}_{jk}(\alpha, \lambda\pm\iota0),~\tilde{c}_{jk}(\alpha, \lambda)$ and $\tilde{d}_{jk}(\alpha, \lambda)$ denote the coefficients of the cofactor matrices corresponding to
$B(\alpha, \lambda\pm\iota0)$, $C(\alpha, \lambda)$ and $D(\alpha, \lambda)$ respectively. 

Define
$F(\alpha, \lambda)
  := \det(B(\alpha,\lambda+\iota0))
    = F_1(\alpha, \lambda) + \iota F_2(\alpha, \lambda),$ where
\begin{equation}\label{F1expression}
    F_1(\alpha, \lambda)
    = \left(1 + \alpha \gamma(\eta_{11}, \lambda)\right)
      \left(1 + \alpha \gamma(\eta_{22}, \lambda)\right)
      - \alpha^2 \left|\gamma(\eta_{12}, \lambda)\right|^2-\alpha^2\pi^2(\eta_{11}(\lambda)\eta_{22}(\lambda)-|\eta_{12}(\lambda)|^2),
\end{equation}
and
\begin{equation}\label{F2expression}
    F_2(\alpha, \lambda)=\alpha\pi\sum_{j,k=1}^2\tilde{c}_{jk}(\alpha,\lambda)\eta_{kj}(\lambda).
\end{equation}
Note that, by the Birman--Schwinger principle (see for example \cite[Section 1.2.8]{FrankLaptevWeidl}), $H_{\alpha}$ has a discrete
eigenvalue $\mu<0$ if and only if $0$ is an eigenvalue of $B(\alpha,\mu+\iota0)$.
The following result provides a necessary and sufficient condition  $\mu\geq{0}$ to be an embedded eigenvalue of $H_{\alpha}$. The proof follows
the arguments of~\cite[Proposition~10.17]{KBSBook}.

\begin{theorem}\label{thm-eigenvalue}
A real number $\mu\geq 0$ is an eigenvalue of $H_{\alpha}$ if and only if $0$ is an eigenvalue of $B(\alpha,\mu+\iota0)$. Moreover, the eigenspaces \text{Ker}$(B(\alpha,\mu+\iota0))$ and $\text{Ker}(H_{\alpha}-\mu)$ are linearly isomorphic.   

\end{theorem} \begin{proof}Firstly assume $\mu>0.$
Suppose $\xi\in\text{Ker}(B(\alpha,\mu+\iota0)).$ Then, by ~\eqref{eqn-laplacianresolventlimit},
\begin{align*}
    ||(\tau^*\xi)_{\mu}||^2
   & = \frac{1}{\pi}\lim_{\epsilon\to 0^+}
      \Im \langle R_0(\mu + \iota\epsilon)\tau^*\xi, \tau^*\xi \rangle=\frac{1}{\pi}\lim_{\epsilon\to 0^+}
       \Im \langle \tau R_0(\mu + \iota\epsilon)\tau^*\xi, \xi \rangle\\&= \frac{1}{\alpha\pi}\Im \langle (B(\alpha,\mu+\iota0)-I)\xi, \xi \rangle
     = \frac{1}{\alpha\pi}\Im \langle -\xi, \xi \rangle = 0.
\end{align*}
Hence $(\tau^{*}\xi)_{\mu}=0$. Since $\tau^{*}\xi \in \mathcal{E}$, it follows from
Lemma~\ref{derivative}\,(a) that
$\Phi_{\tau^{*}\xi,\mu} \in \mathcal{E}$ and therefore
$\Phi_{\tau^{*}\xi,\mu} \in \mathcal{D}(H_{\alpha})$.
Furthermore, by Lemma~\ref{derivative} (b), we have
\[
    R_0(\mu + \iota\epsilon)\tau^*\xi \to \Phi_{\tau^*\xi,\mu}
    \quad \text{in } L^2(\RR^3) \text{ as } \epsilon \to 0^+,
\]
which implies that
\[
    \tau R_0(\mu + \iota\epsilon)\tau^*\xi \to \tau\Phi_{\tau^*\xi,\mu}.
\]
Hence $\tau\Phi_{\tau^*\xi,\mu} = \frac{1}{\alpha}(B(\alpha,\mu+\iota0)-I)\xi = -\tfrac{1}{\alpha}\xi$ and therefore
\[
    (H_{\alpha} - \mu)\Phi_{\tau^*\xi,\mu}
    = (H_0 - \mu)\Phi_{\tau^*\xi,\mu} + \alpha V\Phi_{\tau^*\xi,\mu}
    = \tau^*\xi + \alpha \tau^*\tau \Phi_{\tau^*\xi,\mu}=0.
\]
Conversely, suppose $f\in\text{Ker}(H_\alpha-\mu)$.  
Then
\[
    (\lambda - \mu)f_\lambda + \alpha(Vf)_\lambda = 0
    \quad \text{for a.e. } \lambda \in (0,\infty),
\]
which implies
\begin{equation}\label{eq101}
     f_\lambda = -\frac{\alpha(Vf)_\lambda}{\lambda - \mu}  \quad \text{for a.e. } \lambda \in (0,\infty).
\end{equation}
Since $f \in L^2(\RR^3)$, the function $\lambda\mapsto f_\lambda$ is in $L^2((0,\infty),L^2(S^2))$ which implies $(Vf)_{\mu} = 0$.  
By Lemma~\ref{derivative}~(b) and \eqref{eq101},
\[
    \lim_{\epsilon\to0^+} R_0(\mu + \iota\epsilon)Vf
    = \Phi_{Vf,\mu}=-\frac{1}{\alpha}f.
\]

This implies $\tau f\in$Ker$(B(\alpha,\mu+\iota0)).$ Finally, note that
\[
   T:\text{Ker}(B(\alpha,\mu+\iota0))\to\text{Ker}(H_{\alpha}-\mu),\quad T \xi = \Phi_{\tau^*\xi,\mu},
   \]is a linear isomorphism with $T^{-1} f = -\alpha\tau f.$ 
   
   It remains to consider the case $\mu=0$. The argument proceeds as above, noting
that, since $\tau^{*}\xi \in \mathcal{E}$, we have $(\tau^{*}\xi)_{0}=0$ and hence
$\Phi_{\tau^{*}\xi,0} \in \mathcal{E}$. Similarly, since $u_{1},u_{2}\in\mathcal{E}$,
it follows that $Vf \in \mathcal{E}$ and therefore $(Vf)_{0}=0$. The remainder of
the proof follows in the same way.
This completes the proof.

\end{proof}

\bcor\label{Cor-Modeleigenvalue}
Suppose that $B(\alpha_0,\lambda_0+\iota0)=0$ for some $\alpha_0\in\RR$ and $\lambda_0\geq 0$. Then:
\begin{enumerate}[label=(\alph*)]
    \item $\lambda_0$ is an eigenvalue of $H_{\alpha_0}$ of multiplicity two and the corresponding eigenspace is generated by
   $\phi_1 = \Phi_{u_1,\lambda_0}$ and $ 
        \phi_2 = \Phi_{u_2,\lambda_0}$.
    \item The eigenfunctions satisfy
    \[
        \langle u_j, \phi_k \rangle = -\frac{\delta_{jk}}{\alpha_0}, \qquad j,k \in \{1,2\}.
    \]
\end{enumerate}
\ecor

\begin{proof}
By hypothesis, we have 
\[
u_{1,\lambda_0}= u_{2,\lambda_0} = \delta_{jk} + \alpha_0\, \gamma(\eta_{jk}, \lambda_0) = 0.
\]
Hence, part~(a) follows directly from Theorem~\ref{thm-eigenvalue}.  
Furthermore, since $\gamma(\eta_{jk}, \lambda_0) = \langle u_j, \phi_k \rangle$ for $j,k \in \{1,2\}$, part~(b) follows immediately.
\end{proof} 

\begin{remark}
    Note that if $\alpha_0>0$, then $H_{\alpha_0}$ is a strictly positive operator and hence $0$ cannot be a threshold eigenvalue. Therefore, for $\lambda_0=0$ to be a possible threshold eigenvalue, it is necessary that $\alpha_0<0$.
\end{remark}

\paragraph{\textbf{Calculation of Hessian and its determinant:}} To study the spectral behaviour of $H_\alpha$ near the degenerate eigenvalue $\lambda_0$ of $H_{\alpha_0}$, we analyze the zeros of $F_1(\alpha,\cdot)$. By Lemma~\ref{derivative}\,(c), $\frac{\partial F_1}{\partial\lambda}(\alpha_0,\lambda_0)=\frac{\partial F_1}{\partial\alpha}(\alpha_0,\lambda_0)=0$, so we further look at the Hessian, $D^2{F_1}$, at $(\alpha_0,\lambda_0).$ 
Again by Lemma~\ref{derivative}\,(c), we have
\begin{align}\label{eqn-lambdaderivative}
 \frac{\partial^2 F_1}{\partial\lambda^2}(\alpha_0,\lambda_0)
    &= 2\alpha_0^2\bigl(\|\phi_1\|^2\|\phi_2\|^2 - |\langle \phi_1,\phi_2\rangle|^2\bigr), \\[6pt]\label{eqn-1111}
 \frac{\partial^2 F_1}{\partial\alpha^2}(\alpha_0,\lambda_0)
    &= \frac{2}{\alpha_0^2},\ \frac{\partial^2 F_1}{\partial\alpha\,\partial\lambda}(\alpha_0,\lambda_0)
   = -\bigl(\|\phi_1\|^2 + \|\phi_2\|^2\bigr).
\end{align}
Thus the Hessian, $D^2F_1$, at point $(\alpha_0,\lambda_0)$ is invertible when it has non-zero determinant
\begin{equation}\label{eqn-Hessian}
   \det \left(D^2{F_1}(\alpha_0,\lambda_0)\right)
   = -\bigl(\|\phi_1\|^2 - \|\phi_2\|^2\bigr)^2
     - 4\,|\langle \phi_1,\phi_2\rangle|^2
\end{equation} 
 \noindent In the following lemma, we discuss the zeros of $F_1(\alpha,\cdot)$ near $\alpha=\alpha_0.$

\begin{lemma}\label{resonance}
Suppose that $B(\alpha_0,\lambda_0+\iota0)=0$ and $\det \left(D^2{F_1}(\alpha_0,\lambda_0)\right)<0$ for some $\alpha_0\in\RR$ and $\lambda_0\geq0$.
Then there exist open intervals $J$ around $\lambda_0$ and $I$ around $\alpha_0$ and two distinct smooth functions 
\[
\lambda_j : I \to J, \quad j=1,2,
\]
such that $\lambda_j(\alpha_0) = \lambda_0$ and 
$
F_1(\alpha, \lambda_j(\alpha)) = 0 \quad \text{for all } \alpha \in I
$
with \begin{equation}\label{eqn-derivativelambda}
\lambda_j'(\alpha_0)
   = \frac{\|\phi_1\|^2 + \|\phi_2\|^2 + (-1)^{j+1} d}
           {2\alpha_0^2\bigl(\|\phi_1\|^2\|\phi_2\|^2 - |\langle \phi_1,\phi_2\rangle|^2\bigr)},
\end{equation}where $d=\sqrt{-\det \left(D^2{F_1}(\alpha_0,\lambda_0)\right)}.$
If $\lambda_0>0$ the interval $J$  can be chosen to be in $(0,\infty)$ and henceforth $J$ will be assumed to be so. 
\end{lemma}

\begin{remark}
    It is readily checked that $\lambda'_j(\alpha_0)>{0}$ for $j=1,2$, using the fact that $\phi_1,\phi_2$ are not linearly dependent,  and hence $\lambda_1, \lambda_2$ are increasing functions in a neighbourhood of $\alpha_0$. 
\end{remark}

\begin{proof}[Proof of Lemma~$\ref{resonance}$]
Since $\phi_1$ and $\phi_2$ are linearly independent, by \eqref{eqn-lambdaderivative}, $\frac{\partial^2F_1}{\partial\lambda^2}(\alpha_0,\lambda_0)\neq0$. Moreover as $\det(D^2{F_1})(\alpha_0,\lambda_0)<0$, by an application of Morse lemma (Corollary \ref{Cor-Morse}), there exists open intervals $I$ around $\alpha_0$ and $J$ around $\lambda_0$ and two distinct smooth functions $\lambda_j:I\to J$ such that $\lambda_j(\alpha_0) = \lambda_0$ and 
\[
F_1(\alpha, \lambda_j(\alpha)) = 0 \quad \text{for all } \alpha \in I 
\]  which proves part (a). Differentiating $F_1(\alpha,\lambda_j(\alpha))=0$ twice with respect to $\alpha$, we obtain
\[
\frac{\partial^2 F_1}{\partial\alpha^2}(\alpha,\lambda_j(\alpha))
 + 2\frac{\partial^2 F_1}{\partial\alpha\partial\lambda}(\alpha,\lambda_j(\alpha))\lambda_j'(\alpha)
 + \frac{\partial^2 F_1}{\partial\lambda^2}(\alpha,\lambda_j(\alpha))\lambda_j'(\alpha)^2
 + \frac{\partial F_1}{\partial\lambda}(\alpha,\lambda_j(\alpha))\lambda_j''(\alpha)
 = 0.
\]
Since $\dfrac{\partial F_1}{\partial\lambda}(\alpha_0,\lambda_0)=0$, evaluating the above at $\alpha=\alpha_0$ yields
\begin{equation}\label{lambdapath}
\frac{\partial^2 F_1}{\partial\alpha^2}(\alpha_0,\lambda_0)
 + 2\frac{\partial^2 F_1}{\partial\alpha\partial\lambda}(\alpha_0,\lambda_0)\lambda_j'(\alpha_0)
 + \frac{\partial^2 F_1}{\partial\lambda^2}(\alpha_0,\lambda_0)\lambda_j'(\alpha_0)^2
 = 0.
\end{equation}
Solving the quadratic equation~\eqref{lambdapath} in $\lambda_j'(\alpha_0)$ gives \eqref{eqn-derivativelambda}. 
\end{proof}
 
 We next analyze the spectral behavior of $H_\alpha$ near $\lambda_0>0$ under the assumptions of Lemma~\ref{resonance}. Depending on whether $F_2(\alpha,\lambda_j(\alpha))$ vanishes or not, we obtain either  the existence of a simple embedded eigenvalue or absolute continuity of the spectrum  as stated in the following theorem. 
 
 The threshold eigenvalue case $\lambda_0=0$ is discussed in Theorem~\ref{Threshold}.
 \begin{theorem}\label{Thm-absolutecontinuity}
Assume that the hypotheses of Lemma~$\ref{resonance}$ hold, let $I,J$ and
$\lambda_j(\alpha)$, $j=1,2$, be as in that lemma and assume $\lambda_0>0.$
 Then the following
statements hold:
\begin{enumerate}[label=(\alph*)]
    \item For $j=1,2$ if $F_{2}(\alpha,\lambda_{j}(\alpha))\neq 0$ for all
    $\alpha\in I\setminus\{\alpha_0\}$, then for every
    $\alpha\in I\setminus\{\alpha_0\}$ the spectrum of $H_\alpha$ in $J$ is purely 
    absolutely continuous and the spectrum of $H_{\alpha_0}$ in
    $J\setminus\{\lambda_0\}$ is absolutely continuous.

    \item If there exist $\alpha\in I\setminus\{\alpha_0\}$ and $j\in\{1,2\}$ such
    that $F_{2}(\alpha,\lambda_j(\alpha))=0$, then $\lambda_j(\alpha)$ is a simple
    eigenvalue of $H_\alpha$ with corresponding eigenvector $\Phi_{v_j,\lambda_j(\alpha)}$ where $v_j=u_2(\lambda_j(\alpha))u_1-u_1(\lambda_j(\alpha))u_2$.
\end{enumerate}
\end{theorem}

\begin{proof}Assume that $F_{2}(\alpha,\lambda_{j}(\alpha))\neq 0$ for all
    $\alpha\in I\setminus\{\alpha_0\}$ and $j=1,2$.
To prove that the spectrum of $H_\alpha$ in $J$ is
    absolutely continuous, note that since $\mathcal E$ is dense in $L^2(\RR^3)$, it suffices by Corollary~3.28 in \cite{Teschl}, 
to show that 
\(\lim\limits_{\epsilon \to 0^{+}}\langle R_{\alpha}(\lambda + \iota \epsilon)v,w\rangle\) 
exists for all \(\lambda \in J\) and $v,w\in\mathcal E$. Using Lemma \ref{newlemma} and \eqref{eqn-laplacianresolventlimit}, this reduces to proving that
\[
F(\alpha,\lambda) \neq 0, 
\qquad \forall\, \lambda \in J, \ \alpha \in I \setminus \{\alpha_{0}\}.
\]
By Lemma~\ref{resonance}, for $\alpha \in I \setminus \{\alpha_0\}$,
\[
F_1(\alpha,\lambda) \neq 0 \quad \text{for } \lambda \in  J \setminus \{\lambda_1(\alpha),\lambda_2(\alpha)\},
\]
which implies \[
F(\alpha,\lambda) \neq 0 \quad \text{for } \lambda \in J \setminus \{\lambda_1(\alpha),\lambda_2(\alpha)\}.
\]
Moreover, since \[
F_2(\alpha,\lambda_j(\alpha)) \neq 0, \quad \forall\, \alpha \in I \setminus \{\alpha_0\},\,j=1,2,
\]
and hence for all $\alpha \in I \setminus \{\alpha_0\}$, the spectrum of $H_\alpha$ is purely absolutely continuous in $ J.$ 

\vspace{0.5em}
Finally, note that $F_1(\alpha_0,\lambda) \neq 0$ for $\lambda \in  J \setminus \{\lambda_0\}$. 
This implies that the operator $H_{\alpha_0}$  is absolutely continuous in $J\setminus \{\lambda_0\}$. This completes the proof of part (a).

To prove part (b), note that if \(F_{2}(\alpha,\lambda_{j}(\alpha))=0\),  
then \(F(\alpha,\lambda_{j}(\alpha))=0\) and consequently 
\(\operatorname{Rank}(B(\alpha,\lambda_{j}(\alpha)+\iota0))=1\). Now the statement  follows from Theorem~\ref{thm-eigenvalue}.

\end{proof}


\begin{theorem}[\textbf{Threshold case}]\label{Threshold}
Assume that the hypotheses of Lemma~$\ref{resonance}$ hold and assume $\lambda_0=0.$
 Then the following
statements hold:
\begin{enumerate}[label=(\alph*)]
\item For $\alpha<\alpha_0$, $\lambda_1(\alpha)$ and $\lambda_2(\alpha)$ are discrete simple eigenvalues of $H_\alpha.$
    \item For $\alpha>\alpha_0$, $\lambda_1(\alpha)$ and $\lambda_2(\alpha)$ are both positive. For $j=1,2$  if $F_{2}(\alpha,\lambda_{j}(\alpha))\neq 0$ for all $\alpha>\alpha_0$ then the spectrum of $H_\alpha$ in $J\cap [0,\infty)$ is
    absolutely continuous. Moreover, the spectrum of $H_{\alpha_0}$ in
    $J\cap (0,\infty)$ is absolutely continuous.
 \item If there exists $\alpha>\alpha_0$ such
    that $F_{2}(\alpha,\lambda_j(\alpha))=0$ for some $j\in\{1,2\}$ , then $\lambda_j(\alpha)$ is a simple
    eigenvalue of $H_\alpha$ with corresponding eigenvector $\Phi_{v_j,\lambda_j(\alpha)}$ where $v_j=u_2(\lambda_j(\alpha))u_1-u_1(\lambda_j(\alpha))u_2$.
\end{enumerate}
\end{theorem}
\begin{proof}
     Since
$\lambda_j'(\alpha_0)>0$,
we have $\lambda_j(\alpha)<0$ for $j=1,2$ and all $\alpha\in I$. Moreover, since
$F_2(\alpha,\lambda)=0$ for $\lambda<0$, it follows from the Birman--Schwinger
principle that $\lambda_j(\alpha),\,j=1,2$ are simple discrete eigenvalues of $H_\alpha$ for all
$\alpha<\alpha_0$ with $\alpha\in I$ which proves (a).  To prove (b), note that $\lambda_j(\alpha)>0$ for $j=1,2$ and all
$\alpha>\alpha_0$ with $\alpha\in I$. Then the argument proceeds in the same way as in Theorem~\ref{Thm-absolutecontinuity} (a). The proof of (c) is similar to that of Theorem~\ref{Thm-absolutecontinuity} (b).
\end{proof}

\section{The model}\label{5}

Based on the results in the previous section, in particular Corollary~\ref{Cor-Modeleigenvalue}, 
Lemma~\ref{resonance} and Theorem~\ref{Thm-absolutecontinuity},  we describe the model (see the definition below) \[H_\alpha=H_0+\alpha V\quad\text{ on }L^2(\RR^3)\]such that $\lambda_0$ is a multiplicity two embedded eigenvalue of $H_{\alpha_0}$ with eigenspace spanned by $\{\phi_1,\phi_2\}$ where $\phi_1=\Phi_{u_1,\lambda_0},\phi_2=\Phi_{u_2,\lambda_0}$. For $\alpha\neq\alpha_0$, $H_\alpha$ is purely absolutely continuous in an interval around $\lambda_0.$ In the forthcoming sections, we study the resonance phenomenon near the embedded eigenvalue $\lambda_0$ by analyzing the asymptotic behaviour of spectral density, scattering amplitude and time delay as $\alpha\to\alpha_0.$  
\begin{definition}[The model]\label{model2} Let $\alpha_0\in\RR$ and $\lambda_0>0$ be fixed. For $\alpha\in\RR$, consider the operator \begin{equation}\label{model22}
H_\alpha=H_0+\alpha V\quad\text{ on }L^2(\RR^3)
\end{equation} and $u_j,\,j=1,2$ satisfy the following conditions:\begin{enumerate}[label=(\alph*)]
    \item $u_1,u_2\in\mathcal E$ and $u_1\perp u_2$.
    \item $u_{1,\lambda_0}=u_{2,\lambda_0}=0,\quad\delta_{jk}+\alpha_0\gamma(\eta_{jk},\lambda_0)=0$ $($equivalently $B(\alpha_0,\lambda_0+\iota0)=0$$)$.
    \item $\langle \phi_1, \phi_2 \rangle \neq 0$ or $\|\phi_1\| \neq \|\phi_2\|$ $($equivalently the determinant $\det \left(D^2F_1(\alpha_0,\lambda_0)\right)<0$$)$.
    \item The function $F_2(\alpha,\lambda_j(\alpha))\neq 0$ for all $\alpha\in I$ and for $j=1,2$ where $I$ and $\lambda_j(\alpha)$ are as obtained in Lemma $\ref{resonance}$.
\end{enumerate}
Without loss of generality, we assume that $||\phi_1||\leq ||\phi_2||$ which can be obtained otherwise by swapping $u_1$ and $u_2.$
\end{definition}

The operator \(H_{\alpha_{0}}\) has an embedded eigenvalue \(\lambda_{0}\) with eigenspace  \(\operatorname{span}\{\phi_{1},\phi_{2}\}\).  
To analyse  the resonance phenomenon for  $H_\alpha$ near the embedded eigenvalue \(\lambda_{0}=\lambda(\alpha_0),\)  along the paths 
\(\lambda_{1}(\alpha)\) and \(\lambda_{2}(\alpha)\), we will choose a suitable
orthonormal basis \(\{\psi_{1},\psi_{2}\}\) of this eigenspace. 

The next lemma gives bounds for $\lambda_j'(\alpha_0)$. Its proof is given in the appendix.\begin{lemma}\label{5.2} We have
      \begin{equation}
    \lambda_1'(\alpha_0)\geq\frac{1}{\alpha_0^2||\phi_j||^2},\ \lambda_2'(\alpha_0)\leq\frac{1}{\alpha_0^2||\phi_j||^2} \quad\text{for }j=1,2. \end{equation}
    and \begin{equation}
        \label{iff2}\lambda_1'(\alpha_0)=\frac{1}{\alpha_0^2||\phi_1||^2},\  \lambda_2'(\alpha_0)=\frac{1}{\alpha_0^2||\phi_2||^2}\iff \phi_1\perp\phi_2 
    \end{equation}where $\lambda_j(\alpha)$'s are as obtained in Lemma $\ref{resonance}$.
\end{lemma}

\paragraph{\textbf{Construction of the  canonical eigenbasis at $\lambda_0$:}}
Define
\begin{equation}\label{psi1psi2}
    \psi_1 = a_{11}\phi_1 + a_{12}\phi_2,
    \quad \text{and} \quad
    \psi_2 = a_{21}\phi_1 + a_{22}\phi_2,
\end{equation}
where
\begin{align*}
a_{11} &= \sqrt{\frac{-1 + \alpha_0^2 \|\phi_2\|_2^2 \lambda_1'(\alpha_0)}{d}}, 
&\quad
a_{12} &= -q(\langle \phi_1, \phi_2 \rangle)
          \sqrt{\frac{-1 + \alpha_0^2 \|\phi_1\|_2^2 \lambda_1'(\alpha_0)}{d}}, \\[4pt]
a_{21} &= \sqrt{\frac{1 - \alpha_0^2 \|\phi_2\|_2^2 \lambda_2'(\alpha_0)}{d}}, 
&\quad
a_{22} &= q(\langle \phi_1, \phi_2 \rangle)
          \sqrt{\frac{1 - \alpha_0^2 \|\phi_1\|_2^2 \lambda_2'(\alpha_0)}{d}}
\end{align*}
and $q(z)=z/|z|$ denotes the phase factor of $z,$ defined to be $1$ when $z=0.$ For notational simplicity, we denote $(\phi_j)_\lambda$ by $\phi_{j,\lambda}$ and $(\psi_j)_\lambda$ by $\psi_{j,\lambda}$ for $j=1,2.$

In the next lemma, we derive several relations among the coefficients $a_{jk}$
and establish that $\{\psi_1,\psi_2\}$ is an orthonormal set. The proof is deferred
to the appendix.

\blem\label{Lem-aijrelation}
\begin{enumerate}[label=(\alph*)]
\item For $j = 1, 2$, \[a_{j1}a_{j2} = \dfrac{(-1)^j}{d}\,\alpha_0^2\langle \phi_1, \phi_2 \rangle \lambda_j'(\alpha_0).\]
\item For $j,k\in\{1,2\}$, \[\overline{a_{j1}}a_{k1} + \overline{a_{j2}}a_{k2} =\delta_{jk} \alpha_0^2 \lambda_j'(\alpha_0).\]
\item $\|\psi_1\| = \|\psi_2\| = 1$ and $\psi_1 \perp \psi_2.$
\end{enumerate}
\elem

\section{Spectral density and its asymptotic behaviour}\label{6} In our model, since \( H_\alpha \) is purely absolutely continuous in $J$ for \( \alpha\in I\setminus\{ \alpha_0\} \), by Stone's formula \eqref{stone}, the spectral density $ \rho^{v,w}(\alpha,\cdot)$ associated with the measure $\langle E_\alpha(\diff\lambda)v,w\rangle$ for $v,w\in L^2(\RR^3)$ is given by \begin{equation}\label{spectraldensity}\begin{split}
    \rho^{v,w}(\alpha,\lambda)&=\frac{\diff}{\diff\lambda}\langle E_\alpha(\diff\lambda)v,w\rangle\\&=\frac{1}{2\pi\iota}\left(\langle R_\alpha(\lambda+\iota0)v,w\rangle-\langle R_\alpha(\lambda-\iota0)v,w\rangle\right)\end{split}\end{equation} for almost every $\lambda\in J.$ For notational simplicity, we denote $\rho^{v,v}$ by $\rho^v.$ In what follows, for $v,w\in\mathcal E$ we study the asymptotic behaviour of spectral densities \(\rho^{v,w}(\alpha,\cdot)\) in terms of the Cauchy distribution near the embedded eigenvalue
\(\lambda_0\) as \(\alpha \to \alpha_0\)
which leads to spectral concentration near \(\lambda_0\) and yields the
corresponding time-decay behaviour.

\noindent A direct application of Lemma \ref{derivative} yields\begin{equation}\label{eqn-cijderivative}\frac{\diff}{\diff\alpha}\left(c_{jk}(\alpha,\lambda_l(\alpha))\right)\big|_{\alpha=\alpha_0}=-\frac{\delta_{jk}}{\alpha_0}+\alpha_0\langle\phi_k,\phi_j\rangle\lambda_l'(\alpha_0),\quad j,k,l\in\{1,2\}.\end{equation}
For $\alpha\in J\setminus\{\alpha_0\}$ and $l=1,2$, define \begin{equation}\label{eqn-kappa1}\kappa_l(\alpha):=\frac{F_2(\alpha,\lambda_l(\alpha))}{\alpha  \tilde{c}_{ll}(\alpha,\lambda_l(\alpha))||\phi_l||^2} .\end{equation} By \eqref{iff2} and \eqref{eqn-cijderivative}, for $l=1,2$, we have  
\begin{equation}\label{eqn-cll}
\frac{\diff}{\diff\alpha}\big(\tilde{c}_{ll}(\alpha,\lambda_l(\alpha))\big)\big|_{\alpha=\alpha_0}\neq0.
\end{equation}
It then follows that $\kappa_l(\alpha)=O(|\alpha-\alpha_0|^2)$ as $\alpha\to\alpha_0$ for $l=1,2$. 
Moreover, if $u_{j,\lambda}$ vanish to order $n$ at $\lambda=\lambda_0$ for $j=1,2$, then 
\begin{equation}\label{kappaorder}
\kappa_l(\alpha)=O(|\alpha-\alpha_0|^{2n}), \qquad l=1,2.
\end{equation}
For $\alpha\in J$, $l=1,2$, we set $\lambda_l^h(\alpha)=\lambda_l(\alpha)+h\,\kappa_l(\alpha).$\par
Before proving the spectral density asymptotics, we establish the following lemma, which will be used
in the proof and in later part.
\blem\label{Lem-limits}For any fixed $h\in\RR$ and $j,k,l\in\{1,2\}$, we have\begin{enumerate}[label=(\alph*)]\item  \[\lim\limits_{\alpha\to\alpha_0}\frac{\tilde{c}_{jk}(\alpha,\lambda^h_l(\alpha))}{\tilde{c}_{ll}(\alpha,\lambda^h_l(\alpha))}=\frac{\overline{a_{lj}}a_{lk}}{|a_{ll}|^2}.\]
    \item \begin{equation}\label{eqn-F2limit}
        \lim\limits_{\alpha\to\alpha_0}\dfrac{F_2(\alpha,\lambda_l^h(\alpha))}{\tilde{c}_{ll}(\alpha,\lambda_l^h(\alpha))\kappa_l(\alpha)}=\alpha_0||\phi_l||^2.
    \end{equation}\item \begin{equation}\label{eqn-F1limit}
        \lim\limits_{\alpha\to\alpha_0}\dfrac{F_1(\alpha,\lambda_l^h(\alpha))}{\tilde{c}_{ll}(\alpha,\lambda_l^h(\alpha))\kappa_l(\alpha)}=h\frac{\alpha_0}{|a_{ll}|^2}.
    \end{equation} 
\end{enumerate}Furthermore, all the above  limits are uniform with respect to $h$  in compact sets.\elem\begin{proof} The proof of (a) follows by \eqref{eqn-cijderivative} and mean value theorem. To prove (b), first note that \begin{align*}
   \left |\frac{\tilde{c}_{ll}(\alpha,\lambda_l^h(\alpha))}{\tilde{c}_{ll}(\alpha,\lambda_l(\alpha))}-1\right|&\leq\left|\frac{1}{\tilde{c}_{ll}(\alpha,\lambda_l(\alpha))}\int_{\lambda_l(\alpha)}^{\lambda^h_l(\alpha)}\tilde{c}_{ll}'(\alpha,\lambda)\diff\lambda\right|\\&\leq|h|\frac{\kappa_l(\alpha)\sup_{\alpha\in I,\lambda\in\RR}|\tilde{c}_{ll}'(\alpha,\lambda)|}{|\tilde{c}_{ll}(\alpha,\lambda_l(\alpha))|}.
\end{align*}Using \eqref{eqn-cll} and the fact that $\kappa_l(\alpha) = O(|\alpha - \alpha_0|^2)$, we have \begin{equation}\label{eq1}\lim_{\alpha\to\alpha_0}\dfrac{\tilde{c}_{ll}(\alpha,\lambda_l^h(\alpha))}{\tilde{c}_{ll}(\alpha,\lambda_l(\alpha))}=1\qquad \text{uniformly for }h\text{ in compact sets} .\end{equation} By \eqref{eqn-kappa1}, we have
\begin{equation}\label{eq2}\frac{F_2(\alpha,\lambda_l^h(\alpha))}{\tilde{c}_{ll}(\alpha,\lambda_l(\alpha))\kappa_l(\alpha)}=\frac{F_2(\alpha,\lambda_l^h(\alpha))-F_2(\alpha,\lambda_l(\alpha))}{\tilde{c}_{ll}(\alpha,\lambda_l(\alpha))\kappa_l(\alpha)}+\alpha||\phi_l||^2.\end{equation}Now, 
\begin{equation}\label{estimates}\begin{split}
    &\left|F_2(\alpha,\lambda_l^h(\alpha))-F_2(\alpha,\lambda_l(\alpha))-h\kappa_l(\alpha)\frac{\partial F_2}{\partial\lambda}(\alpha,\lambda_l(\alpha))\right|\\&\quad=\left|\int_{\lambda_l(\alpha)}^{\lambda_l^h(\alpha)}\left(\frac{\partial F_2}{\partial\lambda}(\alpha,\lambda)-\frac{\partial F_2}{\partial\lambda}(\alpha,\lambda_l(\alpha))\right)\diff\lambda\right|\leq h^2\kappa_l(\alpha)^2\left|\left|\frac{\partial^2 F_2}{\partial\lambda^2}\right|\right|_\infty.
\end{split}\end{equation}Since $\kappa_l(\alpha)=O(|\alpha-\alpha_0|^2)$ and  $\frac{\partial F_2}{\partial\lambda}(\alpha,\lambda_l(\alpha))=O(|\alpha-\alpha_0|^2)$ as $\alpha$ goes to $\alpha_0$, it follows that \begin{equation}\label{eq3}\lim_{\alpha\to\alpha_0}\frac{F_2(\alpha,\lambda_l^h(\alpha))-F_2(\alpha,\lambda_l(\alpha))}{\tilde{c}_{ll}(\alpha,\lambda_l(\alpha))\kappa_l(\alpha)}=\lim_{\alpha\to\alpha_0}\frac{h\frac{\partial F_2}{\partial\lambda}(\alpha,\lambda_l(\alpha))}{\tilde{c}_{ll}(\alpha,\lambda_l(\alpha)) }=0.\end{equation}

Combining \eqref{eq1},\eqref{eq2} and \eqref{eq3}, we get \eqref{eqn-F2limit}. By \eqref{estimates} and \eqref{eq1}, it follows that the limit in \eqref{eqn-F2limit} is uniform for $h$  in compact sets. Since $F_1(\alpha,\lambda_l(\alpha))=0,$ by similar estimates as in \eqref{estimates}, we get
\begin{equation}\label{221}\lim_{\alpha\to\alpha_0}\frac{F_1(\alpha,\lambda_l^h(\alpha))}{\tilde{c}_{ll}(\alpha,\lambda_l(\alpha))\kappa_l(\alpha)}=\lim_{\alpha\to\alpha_0}\frac{F_1(\alpha,\lambda_l^h(\alpha))-F_1(\alpha,\lambda_l(\alpha))}{\tilde{c}_{ll}(\alpha,\lambda_l(\alpha))\kappa_l(\alpha)}=\lim_{\alpha\to\alpha_0}\frac{h\frac{\partial F_1}{\partial\lambda}(\alpha,\lambda_l(\alpha))}{\tilde{c}_{ll}(\alpha,\lambda_l(\alpha))}\end{equation}uniformly for $h$ in compact sets. 
Now, using part (a),  we get
\begin{align*}\lim_{\alpha\to\alpha_0}\dfrac{\frac{\partial F_1}{\partial\lambda}(\alpha,\lambda_l(\alpha))}{\tilde{c}_{ll}(\alpha,\lambda_l(\alpha))}&=\lim_{\alpha\to\alpha_0}\frac{c_{11}(\alpha,\lambda_l(\alpha))}{\tilde{c}_{ll}(\alpha,\lambda_l(\alpha))}\frac{\partial c_{22}}{\partial\lambda}(\alpha_0,\lambda_0)+\lim_{\alpha\to\alpha_0}\frac{c_{22}(\alpha,\lambda_l(\alpha))}{\tilde{c}_{ll}(\alpha,\lambda_l(\alpha))}\frac{\partial c_{11}}{\partial\lambda}(\alpha_0,\lambda_0)\\&\qquad-2\Re\left[\lim_{\alpha\to\alpha_0}\frac{c_{21}(\alpha,\lambda_l(\alpha))}{\tilde{c}_{ll}(\alpha,\lambda_l(\alpha))}\frac{\partial \overline{c_{21}}}{\partial\lambda}(\alpha_0,\lambda_0)\right]\\&=\alpha_0 ||\phi_2||^2\frac{|a_{l2}|^2}{|a_{ll}|^2}+\alpha_0 ||\phi_1||^2\frac{|a_{l1}|^2}{|a_{ll}|^2}+2\alpha_0\Re\left[\overline{\langle\phi_1,\phi_2\rangle}\frac{\overline{a_{l1}}a_{l2}}{|a_{ll}|^2}\right]\\&=\alpha_0\frac{|a_{l1}|^2||\phi_1||^2+|a_{l2}|^2||\phi_2||^2+2\Re[\overline{\langle\phi_1,\phi_2\rangle }\overline{a_{l1}}a_{l2}]}{|a_{ll}|^2}.\end{align*} Since \[1=||\psi_l||^2=|a_{l1}|^2||\phi_1||^2+|a_{l2}|^2||\phi_2||^2+2\Re[\overline{\langle\phi_1,\phi_2\rangle }\overline{a_{l1}}a_{l2}],\]we obtain\begin{equation}\label{222}
 \lim_{\alpha\to\alpha_0}\dfrac{\frac{\partial F_1}{\partial\lambda}(\alpha,\lambda_l(\alpha))}{\tilde{c}_{ll}(\alpha,\lambda_l(\alpha))}   =\dfrac{\alpha_0}{|a_{ll}|^2}.
\end{equation}
Then $(c)$ follows by substituting \eqref{222} into \eqref{221}.
\end{proof}
The following theorem describes the asymptotic behaviour of the spectral density of 
\(H_{\alpha}\) near \(\alpha_{0},\) after suitable translation
and scaling of “energy” $\lambda,$ in terms of the Cauchy distribution.  
It is one of the principal results of this paper and as consequences of this, we obtain results on spectral concentration and time decay in Section \ref{7}.

\bthm \label{psidensityasymThm} For any fixed $h\in\RR$ and $v,w\in\mathcal E$, we have \begin{equation}\label{uidensityasy}\lim_{\alpha\to\alpha_0}\kappa_l(\alpha)\rho^{v,w}(\alpha,\lambda_l^h(\alpha))=\frac{\langle P_{\psi_l}v,w\rangle}{\pi}\frac{||a_{ll}\phi_l||^2}{h^2+||a_{ll}\phi_l||^4} \end{equation} where $P_{\psi_l}$ denotes the orthogonal projection onto the subspace spanned by $\psi_l$. In particular, for the eigenvectors $\psi_1$ and $\psi_2$ of $H_{\alpha_0}$, we have
      \begin{equation}\label{psidensityasym}\lim_{\alpha\to\alpha_0}\kappa_l(\alpha)\rho^{\psi_j,\psi_k}(\alpha,\lambda_l^h(\alpha))=\frac{\delta_{lj}\delta_{lk}}{\pi}\frac{||a_{ll}\phi_l||^2}{h^2+||a_{ll}\phi_l||^4} .\end{equation}Furthermore, all the limits above are uniform with respect to $h$ in compact sets.
     \ethm
\begin{proof}By Lemma \ref{newlemma}, for any $\lambda\in J$ and $\alpha\in I\setminus\{\alpha_0\}$, we have
 \begin{equation}\label{ch6-R4}\begin{split}
       &\kappa_l(\alpha)\langle  R_\alpha(\lambda_l^h(\alpha)+\iota0)v,w\rangle\\ &=\kappa_l(\alpha)\langle R_0(\lambda_l^h(\alpha)+\iota0)v,w\rangle\\&\ \ -\alpha\frac{\kappa_l(\alpha)\tilde{c}_{ll}(\alpha,\lambda_l^h(\alpha))}{F(\alpha,\lambda_l^h(\alpha))}\sum_{j,k=1}^2\frac{\tilde{b}_{jk}(\alpha,\lambda_l^h(\alpha)+\iota0)}{\tilde{c}_{ll}(\alpha,\lambda_l^h(\alpha))}\langle R_0(\lambda_l^h(\alpha)+\iota0)v,u_j\rangle \langle R_0(\lambda_l^h(\alpha)+\iota0)u_k,w\rangle.\end{split}
    \end{equation}
Since $u_{j,\lambda_0}=0$ for any $j$, by \eqref{eqn-laplacianresolventlimit}, it follows that \begin{equation}\label{ch6-R5}
  \lim_{\alpha\to\alpha_0} \langle R_0(\lambda_l^h(\alpha)+\iota0)v,u_j\rangle=\gamma(\eta_{v,u_j},\lambda_0)=\langle v,\phi_j \rangle 
\end{equation} and \begin{equation}\label{ch6-R6}
    \lim_{\alpha\to\alpha_0} \langle R_0(\lambda_l^h(\alpha)+\iota0)u_k,w\rangle=\gamma(\eta_{u_k,w},\lambda_0)=\langle\phi_k,w\rangle.
\end{equation}
   Applying Lemma \ref{Lem-limits} (b) and (c), we also obtain \begin{equation}\label{1-eqn2}
        \lim_{\alpha\to\alpha_0}\dfrac{F(\alpha,\lambda_l^h(\alpha))}{\tilde{c}_{ll}(\alpha,\lambda_l^h(\alpha))\kappa_l(\alpha)}=\frac{\alpha_0(h+\iota||a_{ll}\phi_l||^2)}{|a_{ll}|^2}.
    \end{equation} Furthermore, by Lemma \ref{Lem-limits} (a), it follows that \[\lim_{\alpha\to\alpha_0}\frac{\tilde{b}_{jk}(\alpha,\lambda_l^h(\alpha)+\iota0)}{\tilde{c}_{ll}(\alpha,\lambda_l^h(\alpha))}=\frac{\overline{a_{lj}}a_{lk}}{|a_{ll}|^2}.\]Taking the limit $\alpha\to\alpha_0$ in \eqref{ch6-R4} and substituting the above limit together with \eqref{ch6-R5}, \eqref{ch6-R6} and \eqref{1-eqn2}, we obtain
\begin{equation} \label{ch6-R7}  \begin{split} \lim_{\alpha\to\alpha_0}\kappa_l(\alpha)\langle  R_\alpha(\lambda_l^h(\alpha)+\iota0)v,w\rangle&=-\frac{1}{h+\iota||a_{ll}\phi_l||^2}\sum_{j,k=1}^2\overline{a_{lj}}a_{lk}\langle v,\phi_j \rangle \langle\phi_k,w\rangle\\&=-\frac{\langle v,\psi_l \rangle \langle\psi_l,w\rangle}{h+\iota}=-\frac{\langle P_{\psi_l}v,w \rangle }{h+\iota||a_{ll}\phi_l||^2}.\end{split}\end{equation}
    Similarly we have
    \begin{equation}\label{ch6-R8}\lim_{\alpha\to\alpha_0}\kappa(\alpha)\langle  R_\alpha(\lambda_l^h(\alpha)-\iota0)v,w\rangle=-\frac{\langle P_{\psi_l}v,w \rangle }{h-\iota||a_{ll}\phi_l||^2}.\end{equation}
  Finally, by using \eqref{spectraldensity}, \eqref{ch6-R7} and \eqref{ch6-R8},
  \eqref{uidensityasy} follows. By Lemma \ref{Lem-limits} and, the continuity of $\gamma(\eta_{v,u_j},\cdot)$ and $\gamma(\eta_{u_j,w},\cdot)$ for any $j$, all the limits above are uniform as stated.
\end{proof}

\section{Spectral concentration, time decay and sojourn time}\label{7}
  \subsection{Spectral Concentration}  Next we consider the phenomenon of spectral concentration of $H_\alpha$ near $\lambda_0$ as $\alpha\to\alpha_0.$ We modify the definition of spectral concentration given in~\cite{KatoBook} by
additionally requiring that the concentrating sets to converge to the eigenvalue under
consideration.

 \bdfn[Spectral Concentration] Let $\{T_\alpha\}_{\alpha\in\RR}$ be a family of self-adjoint operators with the associated spectral measures $\left\{E_\alpha\right\}_{\alpha\in\RR}$ and $\lambda_0$ be an eigenvalue of $T_{\alpha_0}.$ We say that the spectrum of $T_\alpha$ is concentrated near $\lambda_0$ as $\alpha\to\alpha_0$ if there exists a family of Borel sets $\{\mathscr{I}_\alpha\}_{\alpha\in\RR}$ such that both $\inf \mathscr{I}_\alpha,\,\sup\mathscr{I}_\alpha$ converge to $\lambda_0$ and the spectral projection $E_\alpha(\mathscr{I}_\alpha)\xrightarrow{} E_{\alpha_0}(\{\lambda_0\})$ strongly, as $\alpha\to\alpha_0.$\par In addition, for $p>0$ if $\lim\limits_{\alpha\to\alpha_0}\frac{| \mathscr{I}_\alpha|}{|\alpha-\alpha_0|^p}=0$, then we say that the spectrum of $\{T_\alpha\}$ is concentrated to order $p$ near $\lambda_0.$
 \edfn 
\begin{lemma}\label{lem-specconc}
For $l=1,2$, let $\{\mathscr{I}_{l,\alpha}\}$ be a family of open intervals symmetric
about $\lambda_l(\alpha)$ such that
\[
\lim_{\alpha\to\alpha_0}
\frac{|\mathscr{I}_{l,\alpha}|}{2\kappa_l(\alpha)}=\infty,
\qquad
|\mathscr{I}_{l,\alpha}|=o(|\alpha-\alpha_0|)
\quad \text{as } \alpha\to\alpha_0.
\]
Then
\[
\lim_{\alpha\to\alpha_0}
\langle E_\alpha(\mathscr{I}_{l,\alpha})\psi_j,\psi_k\rangle
=
\delta_{jk}\,\delta_{jl},
\qquad j,k,l\in\{1,2\}.
\]
\end{lemma}

\begin{proof}
Since $\lambda_1'(\alpha_0)\neq \lambda_2'(\alpha_0)$, it follows that
$\lambda_1(\alpha)\neq \lambda_2(\alpha)$ for all $\alpha$ sufficiently close
to $\alpha_0$, with $\alpha\neq\alpha_0$.
Moreover, since
\[
|\mathscr{I}_{1,\alpha}|=|\mathscr{I}_{2,\alpha}|=o(|\alpha-\alpha_0|)
\quad\text{and}\quad
|\lambda_1(\alpha)-\lambda_2(\alpha)|=O(|\alpha-\alpha_0|)
\quad \text{as } \alpha\to\alpha_0,
\]
the intervals $\mathscr{I}_{1,\alpha}$ and $\mathscr{I}_{2,\alpha}$ are disjoint
for $\alpha$ sufficiently close to $\alpha_0$.
Set
\[
\mathscr{I}_\alpha := \mathscr{I}_{1,\alpha}\cup \mathscr{I}_{2,\alpha}.
\]
For $\alpha$ near $\alpha_0$, we therefore have
\begin{align*}
\langle E_\alpha(\mathscr{I}_\alpha)\psi_j,\psi_j\rangle
&=
\langle E_\alpha(\mathscr{I}_{1,\alpha})\psi_j,\psi_j\rangle
+
\langle E_\alpha(\mathscr{I}_{2,\alpha})\psi_j,\psi_j\rangle \\
&=
\int_{\mathscr{I}_{1,\alpha}} \rho^{\psi_j}(\alpha,\lambda)\,\diff\lambda
+
\int_{\mathscr{I}_{2,\alpha}} \rho^{\psi_j}(\alpha,\lambda)\,\diff\lambda \\
&=
\int_{\mathbb{R}} f_\alpha(h)\,\diff h,
\end{align*}
where
\begin{align*}
f_\alpha(h)
&=
\chi_{\left[-\frac{|\mathscr{I}_{1,\alpha}|}{2\kappa_1(\alpha)},
              \frac{|\mathscr{I}_{1,\alpha}|}{2\kappa_1(\alpha)}\right]}
\kappa_1(\alpha)\,
\rho^{\psi_j}\bigl(\alpha,\lambda_1(\alpha)+h\kappa_1(\alpha)\bigr) \\
&\quad +
\chi_{\left[-\frac{|\mathscr{I}_{2,\alpha}|}{2\kappa_2(\alpha)},
              \frac{|\mathscr{I}_{2,\alpha}|}{2\kappa_2(\alpha)}\right]}
\kappa_2(\alpha)\,
\rho^{\psi_j}\bigl(\alpha,\lambda_2(\alpha)+h\kappa_2(\alpha)\bigr).
\end{align*}
Since
\[
\lim_{\alpha\to\alpha_0}
\frac{|\mathscr{I}_{l,\alpha}|}{2\kappa_l(\alpha)}=\infty,
\qquad l=1,2,
\]
it follows from \eqref{psidensityasym} that, for every fixed $h\in\mathbb{R}$,
\[
\lim_{\alpha\to\alpha_0}
f_\alpha(h)
=
\frac{1}{\pi}\,
\frac{\|a_{jj}\phi_j\|^2 |a_{jj}|^2}
     {h^2 + \|a_{jj}\phi_j\|^4}.
\]
Moreover,
\[
\int_{\mathbb{R}} f_\alpha(h)\,\diff h
=
\langle E_\alpha(\mathscr{I}_\alpha)\psi_j,\psi_j\rangle
\le 1,
\qquad
\frac{1}{\pi}\int_{\mathbb{R}}
\frac{\|a_{jj}\phi_j\|^2 }
     {h^2 + \|a_{jj}\phi_j\|^4}\,\diff h
= 1.
\]
Hence, by Lemma~8.1 in \cite{LS}, we obtain
\begin{equation}\label{eqn-1115}
\lim_{\alpha\to\alpha_0}
\langle E_\alpha(\mathscr{I}_\alpha)\psi_j,\psi_j\rangle
= 1,
\qquad j=1,2.
\end{equation}
By the same argument applied to each individual interval, we also have
\begin{equation}\label{eqn-1116}
\lim_{\alpha\to\alpha_0}
\langle E_\alpha(\mathscr{I}_{l,\alpha})\psi_l,\psi_l\rangle
= 1,
\qquad l=1,2.
\end{equation}
Combining \eqref{eqn-1115} and \eqref{eqn-1116}, we conclude that
\begin{equation}\label{eqn-1117}
\lim_{\alpha\to\alpha_0}
\langle E_\alpha(\mathscr{I}_{l,\alpha})\psi_j,\psi_j\rangle
= \delta_{lj},
\qquad j,l\in\{1,2\}.
\end{equation}
Finally, since $E_\alpha(\mathscr{I}_{l,\alpha})$ is an orthogonal projection,
the Cauchy--Schwarz inequality yields
\[
\bigl|
\langle E_\alpha(\mathscr{I}_{l,\alpha})\psi_j,\psi_k\rangle
\bigr|
=
\bigl|
\langle E_\alpha(\mathscr{I}_{l,\alpha})\psi_j,
       E_\alpha(\mathscr{I}_{l,\alpha})\psi_k\rangle
\bigr|
\le
\langle E_\alpha(\mathscr{I}_{l,\alpha})\psi_j,\psi_j\rangle^{1/2}
\langle E_\alpha(\mathscr{I}_{l,\alpha})\psi_k,\psi_k\rangle^{1/2}.
\]
Using \eqref{eqn-1117}, we obtain for $j\neq k$,
\begin{equation}\label{eqn-1118}
\lim_{\alpha\to\alpha_0}
\langle E_\alpha(\mathscr{I}_{l,\alpha})\psi_j,\psi_k\rangle
= 0.
\end{equation}
This completes the proof.
\end{proof}

 \bthm[Spectral concentration]\label{thm-specconc} Suppose $u_{j,\lambda}$ vanishes to order $n$ at $\lambda=\lambda_0$ for $j=1,2.$ Then for any $p\in[0,2n)$ and $l=1,2$, there exists a family of intervals $\mathscr{I}_{l,\alpha}$ symmetric around $\lambda_l(\alpha)$ such that $|\mathscr{I}_{l,\alpha}|=o(|\alpha-\alpha_0|^p)$ and \[\slim_{\alpha\to\alpha_0}E_\alpha(\mathscr{I}_{l,\alpha})=P_{\psi_l}.\]i=1
 In particular, for any $p\in[0,2n),$ the spectrum of $H_\alpha$ is concentrated near $\lambda_0$ to order $p$ as $\alpha\to\alpha_0$.\ethm
\begin{proof} By \eqref{kappaorder}, $\kappa_l(\alpha)=O(|\alpha-\alpha_0|^{2n})$ as $\alpha\to\alpha_0.$ Hence, for $l=1,2$ and $q\in[1,2n)$, we choose $\{\mathscr{I}_{l,\alpha}\}$ to be a family of intervals symmetric about $\lambda_l(\alpha)$ such that 
\[
\lim_{\alpha\to\alpha_0}\frac{|\mathscr{I}_{l,\alpha}|}{2\kappa_l(\alpha)}=\infty,
\qquad 
|\mathscr{I}_{l,\alpha}|=o(|\alpha-\alpha_0|^q) \quad \text{as } \alpha\to\alpha_0.
\]
By the resolvent identity, for $\Im z\neq0,$ \[\lim_{\alpha\to\alpha_0}||R_\alpha(z)-R_{\alpha_0}(z)||=0,\] hence $H_\alpha\to H_{\alpha_0}$ in the strong resolvent sense. By Proposition \ref{Pro-specconc}, it follows that   \begin{equation}\label{eq}\slim_{\alpha\to\alpha_0}E_\alpha(\mathscr{I}_{l,\alpha})(I-P_{\lambda_0})=0,\qquad l=1,2.\end{equation}

On the other hand, if \(v_{1}\) and \(v_{2}\) are eigenvectors of \(H_{\alpha_{0}}\), 
then by Lemma~\ref{lem-specconc} it follows that\begin{equation}\label{ev}\lim_{\alpha\to\alpha_0}\langle E_\alpha(\mathscr{I}_{l,\alpha})v_1,v_2\rangle=\langle P_{\psi_l}v_1,v_2\rangle.\end{equation}
 Now for any $v,w\in L^2(\RR^3)$, we have \begin{align*}\langle E_\alpha(\mathscr{I}_{l,\alpha})P_{\lambda_0}v,w\rangle&=\langle E_\alpha(\mathscr{I}_{l,\alpha})P_{\lambda_0}v,P_{\lambda_0}w\rangle+\langle E_\alpha(\mathscr{I}_{l,\alpha})P_{\lambda_0}v,(I-P_{\lambda_0})w\rangle\\&=\langle E_\alpha(\mathscr{I}_{l,\alpha})P_{\lambda_0}v,P_{\lambda_0}w\rangle+\langle P_{\lambda_0}v,E_\alpha(\mathscr{I}_{l,\alpha})(I-P_{\lambda_0})w\rangle.\end{align*}
   By \eqref{eq} and \eqref{ev}, it follows that $\langle E_\alpha(\mathscr{I}_{l,\alpha})P_{\lambda_0}v,w\rangle\to\langle P_{\psi_l}v,w\rangle$ for $l=1,2$ as $\alpha\to\alpha_0.$ Since limiting operator is a projection, weak convergence implies strong convergence and we get $E_\alpha(\mathscr{I}_{l,\alpha})P_{\lambda_0}\to P_{\psi_l}$ strongly as $\alpha\to\alpha_0.$ Hence by \eqref{eq}, we conclude that, for $l=1,2$  \begin{equation}\label{1-eqn}\slim_{\alpha\to\alpha_0}E_\alpha(\mathscr{I}_{l,\alpha})=P_{\psi_l}.\end{equation}Set $\mathscr{I}_{\alpha}=\mathscr{I}_{1,\alpha}\cup \mathscr{I}_{2,\alpha}.$ Then as in Lemma \ref{lem-specconc}, the intervals $\mathscr{I}_{1,\alpha}$ and $\mathscr{I}_{2,\alpha}$ are disjoint and hence by \eqref{1-eqn}, we obtain    \[\slim_{\alpha\to\alpha_0}E_\alpha(\mathscr{I}_{\alpha})=\slim_{\alpha\to\alpha_0}E_\alpha(\mathscr{I}_{1,\alpha})+\slim_{\alpha\to\alpha_0}E_\alpha(\mathscr{I}_{2,\alpha})=P_{\psi_1}+P_{\psi_2}=P_{\lambda_0}.\] By the choice of \(\mathscr{I}_{\alpha}\) and the limit above, it follows that for every 
\(p \in [0,2n),\) the spectrum of \(H_{\alpha}\) is concentrated near \(\lambda_{0}\) to order \(p\) 
as \(\alpha \to \alpha_{0}\).\end{proof}
\brmrk{\label{Ialpha}} Suppose that $|\mathscr{I}_{l,\alpha}|=o(|\alpha-\alpha_0|^{2n})$ for $l=1,2$ in Theorem $\ref{thm-specconc}$ and the limit $\lim\limits_{\alpha\to\alpha_0}\frac{|\mathscr{I}_{l,\alpha}|}{2\kappa_l(\alpha)}$ exist. Using the uniform convergence established in Theorem~$\ref{psidensityasymThm}$ on compact sets, 
we obtain that 
$E_\alpha(\mathscr{I}_{\alpha})$ converges weakly to $c_1P_{\psi_1}+c_2P_{\psi_2}$ for some $0<c_1,c_2<1.$ Since the operator $c_1P_{\psi_1}+c_2P_{\psi_2}$ is not a projection, $E_\alpha(\mathscr{I}_{\alpha})$ cannot converge strongly.\par On the other hand, if \(|\mathscr{I}_{l,\alpha}| = o(|\alpha - \alpha_{0}|^{\,q})\) 
for some \(q > 2n\) and \(l = 1,2\), then \(E_{\alpha}(\mathscr{I}_{\alpha})\) 
converges strongly to \(0\).\ermrk
\subsection{Time decay}
We now study the finer structure of the evolution generated by $H_\alpha$
by shifting the mean and rescaling time in a neighborhood of
$\alpha=\alpha_0$.
The following result is an immediate consequence of
\eqref{psidensityasym}.
\begin{corollary}\label{exp}
For $l=1,2$, we have
\[
\lim_{\alpha\to\alpha_0}
\bigl\langle
\e^{-\iota t\frac{H_\alpha-\lambda_l(\alpha)}{\kappa_l(\alpha)}}
\psi_l,\psi_l
\bigr\rangle
=
\e^{-\|a_{ll}\phi_l\|^{2}|t|}
\]
uniformly for $t\in\mathbb{R}$.
\end{corollary}

\begin{proof}
Fix $l\in\{1,2\}$.
By the spectral theorem,
\begin{equation}\label{eq:dyn-rep}
\biggl\langle
\e^{-\iota t\frac{H_\alpha-\lambda_l(\alpha)}{\kappa_l(\alpha)}}
\psi_l,\psi_l
\biggr\rangle
=
\int_{\mathbb{R}}
\e^{-\iota t\frac{\lambda-\lambda_l(\alpha)}{\kappa_l(\alpha)}}
\,\diff\langle E_\alpha(\lambda)\psi_l,\psi_l\rangle .
\end{equation}
Splitting the integral in \eqref{eq:dyn-rep}, we write
\[
\biggl\langle
\e^{-\iota t\frac{H_\alpha-\lambda_l(\alpha)}{\kappa_l(\alpha)}}
\psi_l,\psi_l
\biggr\rangle
=
\int_J
\e^{-\iota t\frac{\lambda-\lambda_l(\alpha)}{\kappa_l(\alpha)}}
\rho^{\psi_l}(\alpha,\lambda)\,\diff\lambda
+
A_\alpha^{l}(t),
\]
where
\[
A_\alpha^{l}(t)
:=
\int_{\mathbb{R}\setminus J}
\e^{-\iota t\frac{\lambda-\lambda_l(\alpha)}{\kappa_l(\alpha)}}
\,\diff\langle E_\alpha(\lambda)\psi_l,\psi_l\rangle .
\]
Note that
\[
|A_\alpha^{l}(t)|
\le
\langle E_\alpha(\mathbb{R}\setminus J)\psi_l,\psi_l\rangle.
\]

By the same argument as in Lemma~\ref{lem-specconc}, 
we obtain
\[
\lim_{\alpha\to\alpha_0}
\langle E_\alpha(  J)\psi_l,\psi_l\rangle
=1 
\]
and hence
\begin{equation}\label{decay-eqn1}
\lim_{\alpha\to\alpha_0} A_\alpha^{l}(t)=0
\qquad\text{uniformly for } t\in\mathbb{R}.
\end{equation}

We now consider the integral over $J$.
Performing the change of variables
$\lambda=\lambda_l(\alpha)+h\kappa_l(\alpha)$, we obtain
\begin{equation}\label{decay-eqn2}
\int_J
\e^{-\iota t\frac{\lambda-\lambda_l(\alpha)}{\kappa_l(\alpha)}}
\rho^{\psi_l}(\alpha,\lambda)\,\diff\lambda
=
\int_{\mathbb{R}}
\e^{-\iota th}\,
\kappa_l(\alpha)\,
\rho^{\psi_l}\bigl(\alpha,\lambda_l(\alpha)+h\kappa_l(\alpha)\bigr)
\,\diff h
=
\sqrt{2\pi}\,\widehat{g_{l,\alpha}}(t),
\end{equation}
where
\[
g_{l,\alpha}(h)
:=
\kappa_l(\alpha)\,
\rho^{\psi_l}\bigl(\alpha,\lambda_l(\alpha)+h\kappa_l(\alpha)\bigr),
\qquad h\in\mathbb{R}.
\]

Define
\[
g_l(h)
:=
\frac{1}{\pi}
\frac{\|a_{ll}\phi_l\|^{2}}
     {h^{2}+\|a_{ll}\phi_l\|^{4}}.
\]
By \eqref{psidensityasym}, we have
$g_{l,\alpha}(h)\to g_l(h)$ for every $h\in\mathbb{R}$ as $\alpha\to\alpha_0$.
Moreover,
\[
\int_{\mathbb{R}} g_{l,\alpha}(h)\,\diff h
=
\int_{\mathbb{R}} \rho^{\psi_l}(\alpha,\lambda)\,\diff\lambda
\le 1,
\qquad
\int_{\mathbb{R}} g_l(h)\,\diff h = 1.
\]
Hence, by Lemma~8.1 in \cite{LS},
\[
g_{l,\alpha} \longrightarrow g_l
\quad \text{in } L^{1}(\mathbb{R})
\quad \text{as } \alpha\to\alpha_0.
\]
Consequently,
\[
\widehat{g_{l,\alpha}} \longrightarrow \widehat{g_l}
\quad \text{uniformly on } \mathbb{R}.
\]
A direct computation yields
\[
\sqrt{2\pi}\,\widehat{g_l}(t)
=
\e^{-\|a_{ll}\phi_l\|^{2}|t|},
\qquad t\in\mathbb{R}.
\]
Combining this with  \eqref{decay-eqn2} and \eqref{decay-eqn1} completes the proof.
\end{proof}
\subsection{Sojourn time and its properties}
We now study the behaviour of the sojourn time for the reduced operator
\[
\widetilde H_\alpha := E_\alpha(J) H_\alpha.
\]Let \(v \in L^2(\mathbb{R}^3)\).
The \emph{sojourn time} of the state \(v\) with respect to the time evolution
generated by the self-adjoint operator \(\widetilde H_\alpha\) is defined by
\[
\tau_\alpha(v)
:=
\int_{-\infty}^{\infty}
\bigl|\langle \e^{-\iota t \widetilde H_\alpha} v, v \rangle\bigr|^2 \,\diff t=\int_{-\infty}^{\infty}
\bigl|\langle \e^{-\iota t   H_\alpha} E_\alpha(J)v, v \rangle\bigr|^2 \,\diff t.
\]
  Note that $\tau_{\alpha_0}(\psi_l)=\infty$ for $l=1,2$. In the next theorem, we show that $\tau_\alpha(\psi_l)<\infty$ for $\alpha\in I\setminus\{\alpha_0\}$ and establish a lower bound for $\tau_\alpha(\psi_l)$ when $\alpha$ is close to $\alpha_0$. 
\begin{theorem}\label{S1} For any $\alpha\neq\alpha_0$, $l=1,2,$ $\tau_{\alpha}(\psi_l)<\infty.$ Furthermore there exists $\delta>0$ such that \[\tau_\alpha(\psi_l)>\frac{1}{4||a_{ll}\phi_l||^2\kappa_l(\alpha)}\ \ \text{for}\ |\alpha-\alpha_0|<\delta.\] \end{theorem}
 \begin{proof}
Fix $l\in\{1,2\}$. The proof for $\tau_\alpha(\psi_l)<\infty$ is similar to that of Theorem 5.1 in \cite{LS}. For $\alpha\in J\setminus\{\alpha_0\}$ and $T\in(0,\infty),$ define \begin{equation}\label{truncated}
 \tau_\alpha(\psi_l,T):=\int_{-T}^T|\langle\e^{-\iota t\tilde{H_\alpha}}\psi_l,\psi_l\rangle|^2\diff t.
\end{equation} Then $\lim\limits_{T\to\infty}\tau_\alpha(\psi_l, T)=\tau_\alpha(\psi_l).$ Replacing $T$ by $T/\kappa_l(\alpha)$ in \eqref{truncated}, we get
\[\kappa_l(\alpha)\tau_\alpha\left(\psi_l,T/\kappa_l(\alpha)\right)=\int_{-T}^T\left|\langle\e^{-\iota t\frac{\tilde{H_\alpha}-\lambda_l(\alpha)}{\kappa_l(\alpha)}}\psi_l,\psi_l\rangle\right|^2\diff t.\]Note that by Corollary \ref{exp} and \eqref{decay-eqn1}, we have \[\lim_{\alpha\to\alpha_0} \langle\e^{-\iota t\frac{\tilde{H_\alpha}-\lambda_l(\alpha)}{\kappa_l(\alpha)}}\psi_l,\psi_l\rangle =\e^{- ||a_{ll}\phi_l||^2|t|} \] and hence \[\lim_{\alpha\to\alpha_0}\left|\langle\e^{-\iota t\frac{\tilde{H_\alpha}-\lambda_l(\alpha)}{\kappa_l(\alpha)}}\psi_l,\psi_l\rangle\right|^2=\e^{-2||a_{ll}\phi_l||^2|t|}.\]Thus for any finite $T$, by the Lebesgue dominated convergence theorem, we get\[\lim_{\alpha\to\alpha_0}\kappa_l(\alpha)\tau_\alpha\left(\psi_l,T/\kappa_l(\alpha)\right)=\int_{-T}^{T}\e^{-2||a_{ll}\phi_l||^2|t|}\diff t=\frac{1-\e^{-2||a_{ll}\phi_l||^2T}}{||a_{ll}\phi_l||^2}.\]Choose $T_0 > 0$ such that $\frac{1-\e^{-2||a_{ll}\phi_l||^2T}}{||a_{ll}\phi_l||^2} > \frac{1}{2||a_{ll}\phi_l||^2}$ and for this $T_0$, we have a $\delta > 0$ such that for $|\alpha - \alpha_0|< \delta$, \[\kappa_l(\alpha) \tau_\alpha\left(\psi_l, \frac{T_0}{\kappa_l(\alpha)}\right) > \frac{1-\e^{-2||a_{ll}\phi_l||^2T}}{||a_{ll}\phi_l||^2} - \frac{1}{4||a_{ll}\phi_l||^2} > \frac{1}{4||a_{ll}\phi_l||^2}.\]
Since $\tau_\alpha(\psi_l) \geq \tau_\alpha\left(\psi_l, \frac{T_0}{\kappa_l(\alpha)}\right)$, it follows that
$\tau_\alpha(\psi_l) > \frac{1}{4||a_{ll}\phi_l||^2\kappa_l(\alpha)}.$\end{proof}
\section{Behaviour of Scattering amplitude and Time delay}\label{8}
 
For the pair of operators $H_0$ and $H_\alpha$ with $\alpha\neq\alpha_0$ in our model, note that $H_\alpha-H_0$ is rank-two operator. Hence by Theorem 6.2.1  in \cite{Yafaev}, the wave operators \[\Omega^{(\alpha)}_{\pm}:= \slim_{t \to \pm \infty} \e^{\iota tH_\alpha } \e^{-\iota t H_0 }\] exist and are complete. Consider the scattering operator $S^{(\alpha)}=\Omega_+^{(\alpha)*} \Omega_-^{(\alpha)}$ corresponding to the pair $(H_0,H_\alpha).$ Since $S^{(\alpha)}$ is a unitary operator which commutes with $H_0,$ it is decomposable with respect to the spectral representation \eqref{iso} of $H_0$, i.e. for any $v\in L^2(\RR^3)$ \[(S^{(\alpha)}v)_\lambda=S_\lambda^{(\alpha)}v_\lambda\qquad\text{for a.e. }\lambda\in(0,\infty).\]For $\lambda\in(0,\infty),$ $S_\lambda^{(\alpha)}$ is a unitary operator on $L^2(S^2)$ and it is called the scattering matrix at ``energy" $\lambda.$ 

\blem For $\alpha\in I\setminus\{\alpha_0\},\,\lambda\in J,$ the scattering matrix associated with the pair $(H_0,H_\alpha)$ at ``energy" $\lambda$ is given by: 
\begin{equation}\label{eqn-scatteringmatrix}
    S^{(\alpha)}_\lambda=I-\frac{2\pi\iota\alpha}{F(\alpha,\lambda)}\left(\sum_{j,k=1}^2\tilde{b}_{jk}(\alpha,\lambda+\iota0)\langle\cdot,u_{j,\lambda}\rangle u_{k,\lambda}\right)
\end{equation} where $\tilde{b}_{jk}(\alpha,\lambda+\iota0)$ denotes the $(j,k)$-th entry of cofactor matrix of $B(\alpha,\lambda+\iota0)$.
\elem
\begin{proof}The proof is similar to that of \cite[Proposition~8.22]{KBSBook}.
\end{proof}
 
\subsection{Scattering amplitude and total scattering cross section}The scattering amplitude operator at ``energy" $\lambda$ is defined by \[R^{(\alpha)}_\lambda:=S^{(\alpha)}_\lambda-I.\]The scattering amplitude operator $R^{(\alpha)}_\lambda$ measures the deviation of the scattering system $(H_0,H_\alpha)$ from the free dynamics at ``energy" $\lambda$. 
The total scattering cross-section $\sigma_\alpha(\lambda)$ at ``energy" $\lambda$ for the scattering system $(H_0,H_\alpha)$ is defined by  (see \cite[equation (7.69)]{KBSBook})\[\sigma_\alpha(\lambda):=\pi\lambda^{-1}||R^{(\alpha)}_\lambda||^2_{2}\]where $||\cdot||_{2}$ denotes the Hilbert-Schmidt norm. Physically, the total cross-section \(\sigma_\alpha(\lambda)\) quantifies the
total probability of scattering for the pair $(H_0,H_\alpha)$ at energy \(\lambda\). 

In the following theorem, we study the asymptotic behaviour of $R^{(\alpha)}_\lambda$ and total scattering cross-section $\sigma_\alpha(\lambda)$ near $\lambda_0 $  as $\alpha$ goes to  $\alpha_0$.

\begin{theorem}
    \label{Rthm}Consider the model in Definition $\ref{model2}$ along with additional hypothesis that $u'_{l,\lambda_0}\neq 0$ for $l=1,2.$ Then for any fixed $h\in\RR,$ \begin{equation}\label{eqn-1113}\lim_{\alpha\to\alpha_0}R^{(\alpha)}_{\lambda_l^h(\alpha)}=-2\iota\frac{||a_{ll}\phi_l||^2}{h+\iota ||a_{ll}\phi_l||^2} \langle\cdot,\frac{\psi_{l,\lambda_0}}{||\psi_{l,\lambda_0}||}\rangle\frac{\psi_{l,\lambda_0}}{||\psi_{l,\lambda_0}||}\end{equation}{in the this comple-Schmidt norm}. Consequently, for any fixed $h\in\RR,$ \begin{equation}\label{eqn-1114}\lim\limits_{\alpha\to\alpha_0}\sigma_\alpha(\lambda_l^h(\alpha))=\dfrac{4\pi}{\lambda_0}\dfrac{||a_{ll}\phi_l||^4}{h^2+||a_{ll}\phi_l||^4}.\end{equation}
\end{theorem}\begin{proof}Using \eqref{eqn-scatteringmatrix}, we can write $R^{(\alpha)}_{\lambda^h_l(\alpha)}$ for $l=1,2$ as \begin{equation}
    R^{(\alpha)}_{\lambda^h_l(\alpha)}=-2\iota I_l(\alpha,h)\left(\sum_{j,k=1}^2I_{jk,l}(\alpha,h)\langle\cdot,\frac{ u_{j,\lambda_l^h(\alpha)}}{||u_{j,{\lambda_l^h(\alpha)}}||}\rangle\frac{ u_{k,\lambda_l^h(\alpha)}}{||u_{k,{\lambda_l^h(\alpha)}}||}\right)
\end{equation}where 
\begin{equation}\label{eqn-I}
    I_l(\alpha,h)
         = \frac{\kappa_l(\alpha)\, \tilde{c}_{ll}(\alpha,\lambda^h_l(\alpha))}
                {F(\alpha,\lambda^h_l(\alpha))},\quad
    I_{jk,l}(\alpha,h)
        = \frac{\alpha\pi\,\tilde{b}_{jk}(\alpha,\lambda^h_l(\alpha)+\iota0)
                \|u_{j,\lambda_l^h(\alpha)}\|\,
                \|u_{k,\lambda_l^h(\alpha)}\|}
                {\kappa_l(\alpha)\, \tilde{c}_{ll}(\alpha,\lambda^h_l(\alpha))}.\end{equation}
For $j,k,l\in\{1,2\}$, a simple computation gives \begin{equation}\label{eqn-1}\lim_{\alpha\to\alpha_0}\langle\cdot,\frac{ u_{j,\lambda_l^h(\alpha)}}{||u_{j,{\lambda_l^h(\alpha)}}||}\rangle\frac{ u_{k,\lambda_l^h(\alpha)}}{||u_{k,{\lambda_l^h(\alpha)}}||}=\left\langle \cdot,\frac{u'_{j,\lambda_0}}{||u'_{j,\lambda_0}||} \right\rangle \frac{u'_{k,\lambda_0}}{||u'_{k,\lambda_0}||}=\left\langle \cdot,\frac{\phi_{j,\lambda_0}}{||\phi_{j,\lambda_0}||} \right\rangle \frac{\phi_{k,\lambda_0}}{||\phi_{k,\lambda_0}||}\end{equation}in the Hilbert-Schimdt norm. Using Lemma \ref{Lem-limits}~(b) and (c), we also obtain\begin{equation}\label{eqn-2}
    \lim_{\alpha\to\alpha_0}  I_l(\alpha,h)=\frac{|a_{ll}|^2}{\alpha_0(h+\iota||a_{ll}\phi_l||^2)}.
\end{equation}    Next, to compute \(\lim\limits_{\alpha\to\alpha_0}I_{jk,l}(\alpha,h)\), note that $\tilde{b}_{jk}(\alpha,\lambda+\iota0)=\tilde{c}_{jk}(\alpha,\lambda)+\iota\tilde{d}_{jk}(\alpha,\lambda)$ and using \eqref{eqn-kappa1} we write \(I_{jk,l}(\alpha,h)\) as \[  I_{jk,l}(\alpha,h)
        = \frac{\alpha^2\pi||\phi_l||^2\,\left(\tilde{c}_{jk}(\alpha,\lambda^h_l(\alpha))+\iota\tilde{d}_{jk}(\alpha,\lambda^h_l(\alpha))\right)
                \|u_{j,\lambda_l^h(\alpha)}\|\,
                \|u_{k,\lambda_l^h(\alpha)}\|\tilde{c}_{ll}(\alpha,\lambda_l(\alpha))}
                {F_2(\alpha,\lambda_l(\alpha))\tilde{c}_{ll}(\alpha,\lambda^h_l(\alpha))}.\]  Since $\tilde{d}_{jk}(\alpha,\lambda^h_l(\alpha))=\alpha\pi(-1)^{j+k}\eta_{jk}(\lambda^h_l(\alpha))=O(|\alpha-\alpha_0|^2)$, we obtain  \begin{equation}\label{eqn-1112}\lim_{\alpha\to\alpha_0}\frac{\tilde{d}_{jk}(\alpha,\lambda^h_l(\alpha))
                \|u_{j,\lambda_l^h(\alpha)}\|\,
                \|u_{k,\lambda_l^h(\alpha)}\|}
                {F_2(\alpha,\lambda_l(\alpha))}=0.\end{equation}
                By using \eqref{F2expression}, we can write \[\frac{F_2(\alpha,\lambda_l(\alpha))}{\alpha\pi\,\tilde{c}_{ll}(\alpha,\lambda^h_l(\alpha))
                \|u_{j,\lambda_l^h(\alpha)}\|\,
                \|u_{k,\lambda_l^h(\alpha)}\|}=\dfrac{\displaystyle\sum_{m,n=1}^2\tilde{c}_{mn}(\alpha,\lambda_l(\alpha))
               \eta_{nm}(\lambda_l^h(\alpha))}{\tilde{c}_{ll}(\alpha,\lambda^h_l(\alpha))
                \|u_{j,\lambda_l^h(\alpha)}\|\,
                \|u_{k,\lambda_l^h(\alpha)}\|}.\] Using Lemma \ref{Lem-limits} (a) and the limit \[\lim\limits_{\alpha\to\alpha_0}\frac{\eta_{nm}(\lambda_l^h(\alpha))}{\|u_{j,\lambda_l^h(\alpha)}\|\,
                \|u_{k,\lambda_l^h(\alpha)}\|}=\frac{\langle  \phi_{n,\lambda_0},\phi_{m,\lambda_0}\rangle}{||\phi_{j,\lambda_0}|| ||\phi_{k,\lambda_0}||}\]
                we obtain\begin{align*}\lim_{\alpha\to\alpha_0}\frac{F_2(\alpha,\lambda_l(\alpha))}{\alpha\pi\,\tilde{c}_{ll}(\alpha,\lambda^h_l(\alpha))
                \|u_{j,\lambda_l^h(\alpha)}\|\,
                \|u_{k,\lambda_l^h(\alpha)}\|}&=\frac{\sum_{m,n=1}^2\overline{a_{lm}}a_{ln}\langle  \phi_{n,\lambda_0},\phi_{m,\lambda_0}\rangle}{|a_{ll}|^2||\phi_{j,\lambda_0}||\,||\phi_{k,\lambda_0}||}\\&=\frac{||\psi_{l,\lambda_0}||^2}{|a_{ll}|^2||\phi_{j,\lambda_0}||\,||\phi_{k,\lambda_0}||}.\end{align*} Again using Lemma \ref{Lem-limits} (a), \eqref{eqn-1112} and the above limit, we get \begin{equation}\label{eqn-3}
                 \lim_{\alpha\to\alpha_0}I_{jk,l}(\alpha,h)   =\frac{\alpha_0||\phi_l||^2\overline{a_{lj}}a_{lk}||\phi_{j,\lambda_0}||\,||\phi_{k,\lambda_0}||\,}{||\psi_{l,\lambda_0}||^2}.
                \end{equation}
    Hence, using \eqref{eqn-1},~\eqref{eqn-2} and \eqref{eqn-3}, we have \begin{align*}\lim_{\alpha\to\alpha_0}R^{(\alpha)}_{\lambda^h_l(\alpha)}&=-2\iota \frac{||a_{ll}\phi_l||^2}{h+\iota ||a_{ll}\phi_l||^2}\frac{\langle\,\cdot\,,a_{l1}  \phi_{1,\lambda_0}+a_{l2}  \phi_{2,\lambda_0}\rangle \left(a_{l1}  \phi_{1,\lambda_0}+a_{l2}  \phi_{2,\lambda_0}\right)}{||\psi_{l,\lambda_0}||^2}\\&=-2\iota\frac{||a_{ll}\phi_l||^2}{h+\iota ||a_{ll}\phi_l||^2} \left\langle\cdot,\frac{\psi_{l,\lambda_0}}{||\psi_{l,\lambda_0}||}\right\rangle\frac{\psi_{l,\lambda_0}}{||\psi_{l,\lambda_0}||}\end{align*}in the Hilbert-Schmidt norm. This completes the proof of \eqref{eqn-1113} and \eqref{eqn-1114} follows trivially from \eqref{eqn-1113}.
\end{proof}
\begin{remark}
When  $u_{j,\lambda},j=1,2$ vanishes up to order $n$ at $\lambda_0$, then for any fixed $h\in\RR,$ \[ \lim_{\alpha\to\alpha_0}R^{(\alpha)}_{\lambda_l^h(\alpha)}=-2\iota\frac{||a_{ll}\phi_l||^2}{h+\iota ||a_{ll}\phi_l||^2} \left\langle\cdot,\frac{\psi^{(n-1)}_{l,\lambda_0}}{||\psi^{(n-1)}_{l,\lambda_0}||}\right\rangle \frac{\psi^{(n-1)}_{l,\lambda_0}}{||\psi^{(n-1)}_{l,\lambda_0}||}\quad\text{in the Hilbert-Schmidt norm}.\]
\end{remark}

\noindent Note that, the kernel of $R^{(\alpha)}_\lambda$ is given by\begin{equation}\label{eqn-amplitudekernel}
    R^{(\alpha)}_\lambda(\omega_1,\omega_2)=-\frac{2\pi\iota\alpha}{F(\alpha,\lambda)}\left(\sum_{j,k=1}^2\tilde{b}_{jk}(\alpha,\lambda)\overline{u_{j,\lambda}(\omega_2)}u_{k,\lambda}(\omega_1)\right)\quad\text{for a.e. }(\omega_1,\omega_2)\in S^2\times S^2.
\end{equation}We now define the scattering amplitude $f(\lambda;\omega^{\text{in}}\to\omega^{\text{out}})$ from the initial direction $\omega^{\text{in}}$ to the final direction $\omega^{\text{out}}$ at ``energy" $\lambda$ by (see \cite[equation (7.48)]{KBSBook})\[f_\alpha(\lambda;\omega^{\text{in}}\to\omega^{\text{out}}):=-2\pi\iota\lambda^{-1/2}R_\lambda^{(\alpha)}(\omega^{\text{out}},\omega^{\text{in}})\] for almost all $\lambda\in(0,\infty)$ and $(\omega^{\text{in}},\omega^{\text{out}})\in S^2\times S^2.$ \\
For the model in Definition \ref{model2}, if furthermore $u_1,u_2\in\mathcal S(\RR^3),$ then \eqref{iso} implies that $u_{1,\lambda}$ and $u_{2,\lambda}$ are smooth on $S^2$ and hence the amplitude $f_\alpha(\lambda;\omega^{\text{in}}\to\omega^\text{out})$ is continuous in $(0,\infty)\times S^2\times S^2.$ We have the following theorem on the convergence of the  amplitude.
 \bthm Consider the model in Definition $\ref{model2}$ and assume that $u_j\in\mathcal S(\RR^3)$ and  $u'_{j,\lambda_0}\neq 0$ in $L^2(S^2)$ for $j=1,2$. Then, for any fixed $h\in\RR$ \[\lim_{\alpha\to\alpha_0}f_\alpha(\lambda_l^h(\alpha);\omega^{\text{in}}\to\omega^{\text{out}})=-\frac{4\pi}{\sqrt{\lambda_0}}\dfrac{\psi_{l,\lambda_0}(\omega^{\text{out}})\overline{\psi_{l,\lambda_0}(\omega^{\text{in}})}}{||\psi_{l,\lambda_0}||^2}\frac{||a_{ll}\phi_l||^2}{h+\iota ||a_{ll}\phi_l||^2}\] 
uniformly for $(\omega^{\text{in}},\omega^{\text{out}})\in S^2\times S^2.$
   
\ethm
\begin{proof}
  By \eqref{eqn-amplitudekernel}, for $l=1,2,$ we have \begin{equation}
    R^{(\alpha)}_{\lambda^h_l(\alpha)}(\omega_1,\omega_2)=-2\iota I_l(\alpha,h)\left(\sum_{j,k=1}^2I_{jk,l}(\alpha,h)\frac{u_{k,\lambda_l^h(\alpha)}(\omega_1) \overline{u_{j,\lambda_l^h(\alpha)}(\omega_2)}}{||u_{k,{\lambda_l^h(\alpha)}}||\,||u_{j,{\lambda_l^h(\alpha)}}||}\right)
\end{equation}for $(\omega_1,\omega_2)\in S^2\times S^2$ where $I_l(\alpha,h)$,  $I_{jk,l}(\alpha,h)$ are as defined in \eqref{eqn-I}. Note that, for $j,l\in\{1,2\}$\begin{align*}
   \left| \frac{u_{j,\lambda_l^h(\alpha)}-u_{j,\lambda_0}}{\lambda_l^h(\alpha)-\lambda_0}(\omega)-u'_{j,\lambda_0}(\omega)\right|&=\frac{1}{|\lambda_l^h(\alpha)-\lambda_0|}\left|\int_{\lambda_0}^{\lambda_l^h(\alpha)}(u'_{j,\lambda}(\omega)-u'_{j,\lambda_0}(\omega))\diff\lambda\right|\\&=\frac{1}{|\lambda_l^h(\alpha)-\lambda_0|}\left|\int_{\lambda_0}^{\lambda_l^h(\alpha)}\int_{\lambda_0}^\lambda u^{(2)}_{j,s}(\omega)\diff s\diff\lambda\right|\\&\leq |\lambda_l^h(\alpha)-\lambda_0|\sup_{\lambda\in(0,\infty)}||u^{(2)}_{j,\lambda}||_{L^\infty(S^2)}.
\end{align*}
This implies, for $j,k,l\in\{1,2\}$, we have\begin{equation}\label{eqn-limit}\lim_{\alpha\to\alpha_0}\frac{u_{k,\lambda_l^h(\alpha)}(\omega_1) \overline{u_{j,\lambda_l^h(\alpha)}(\omega_2)}}{||u_{k,{\lambda_l^h(\alpha)}}||\,||u_{j,{\lambda_l^h(\alpha)}}||}=\frac{u'_{k,\lambda_0}(\omega_1) \overline{u'_{j,\lambda_0}(\omega_2)}}{||u'_{j,\lambda_0}||\,||u'_{k,\lambda_0}||}=\frac{\phi_{k,\lambda_0}(\omega_1) \overline{\phi_{j,\lambda_0}(\omega_2)}}{||\phi_{j,\lambda_0}||\,||\phi_{k,\lambda_0}||}\end{equation}uniformly for $(\omega_1,\omega_2)\in S^2.$ Hence, by \eqref{eqn-2}, \eqref{eqn-3} and \eqref{eqn-limit}, we get\begin{align*}\lim_{\alpha\to\alpha_0} R^{(\alpha)}_{\lambda^h_l(\alpha)}(\omega_1,\omega_2)=-2\iota\frac{||a_{ll}\phi_l||^2}{h+\iota ||a_{ll}\phi_l||^2}\dfrac{\psi_{l,\lambda_0}(\omega_1)\overline{\psi_{l,\lambda_0}(\omega_2)}}{||\psi_{l,\lambda_0}||^2}\end{align*} and the result follows.
\end{proof}
\begin{remark}When  $u_{j,\lambda},j=1,2$ vanishes up to order $n$ at $\lambda_0$, then for any fixed $h\in\RR$, we have
\[\lim_{\alpha\to\alpha_0}f_\alpha(\lambda_l^h(\alpha);\omega^{\text{in}}\to\omega^{\text{out}})=-\frac{4\pi}{\sqrt{\lambda_0}}\dfrac{\psi^{(n-1)}_{l,\lambda_0}(\omega^{\text{out}})\overline{\psi^{(n-1)}_{l,\lambda_0}(\omega^{\text{in}})}}{||\psi^{(n-1)}_{l,\lambda_0}||^2}\frac{||a_{ll}\phi_l||^2}{h+\iota ||a_{ll}\phi_l||^2}.\] \end{remark}

\subsection{Krein's spectral shift function and time delay}
The  Krein's spectral shift function for the pair of operators $(H_0,H_\alpha)$ is defined by\[\xi_\alpha(\lambda)=\frac{1}{\pi}\lim_{\epsilon\to0^+}\arg \det(I+\alpha VR_0(\lambda+\iota\epsilon))\] where `$\arg$' denotes the principal branch of the argument function, see \cite{Yafaev} for more details on spectral shift function. In our model, observe that $I+\alpha V R_0(z)\restriction_{\operatorname{span}\{u_1,u_2\}}$ is unitarily equivalent to $ B(\alpha,z)$ and
$I+\alpha V R_0(z)\restriction_{\{u_1,u_2\}^\perp}=I.$ Consequently, 
\begin{equation}\label{eqn-104}
\xi_\alpha(\lambda)=\frac{1}{\pi}\arg F(\alpha,\lambda).
\end{equation}
By Birman-Krein's formula (see \cite{Yafaev}), for $\alpha\in I\setminus\{\alpha_0\},\,\lambda\in J$, we have
     \begin{equation}\label{eqn-101}
         \det S^{(\alpha)}_\lambda=\e^{-2\pi\iota\xi_\alpha(\lambda)}.
     \end{equation}
     On the other hand 
by generalized Eisenbud-Wigner formula (see \cite{Martin}), the average time delay $\zeta_\alpha (\lambda)$ for the scattering system $(H_0,H_\alpha)$ at ``energy" $\lambda$ is given by:  \[\zeta_\alpha(\lambda)=\frac{\diff}{\diff\lambda}\arg\left(\det S^{(\alpha)}_\lambda\right).\] Thus  we obtain the relation  \begin{equation}\label{TD1}\zeta_\alpha(\lambda)=-2\pi\xi'_\alpha(\lambda).\end{equation}

 In the following theorem, we obtain the behaviour of average time delay $\zeta_\alpha$  near $\lambda_0$ as $\alpha$ goes to $\alpha_0.$  \bthm\label{31} For each fixed $h\in\RR,$ \[ \lim_{\alpha\to\alpha_0}\kappa_l(\alpha)\zeta_\alpha(\lambda_l^h(\alpha))=2\frac{||a_{ll}\phi_l||^2}{h^2+||a_{ll}\phi_l||^4} .\]  \ethm\begin{proof}By \eqref{TD1} and \eqref{eqn-104}, note  that
\begin{equation}\label{eq14}\zeta_\alpha(\lambda)=-2\pi\xi'_\alpha(\lambda)=-2\frac{1}{\pi}\dfrac{F_1(\alpha,\lambda)\frac{\partial F_2}{\partial\lambda}(\alpha,\lambda)-F_2(\alpha,\lambda)\frac{\partial F_1}{\partial\lambda}(\alpha,\lambda)}{|F(\alpha,\lambda)|^2} .\end{equation}
Substituting $\lambda=\lambda_l^h(\alpha)$ in \eqref{eq14} and multiplying it by $\kappa_l(\alpha),$ we get\[\kappa_l(\alpha)\xi_\alpha'(\lambda_l^h(\alpha))=\frac{1}{\pi}\dfrac{\kappa_l(\alpha)F_1(\alpha,\lambda_l^h(\alpha))\frac{\partial F_2}{\partial\lambda}(\alpha,\lambda_l^h(\alpha))}{|F(\alpha,\lambda_l^h(\alpha))|^2}-\frac{1}{\pi}\frac{\kappa_l(\alpha)F_2(\alpha,\lambda_l^h(\alpha))\frac{\partial F_1}{\partial\lambda}(\alpha,\lambda_l^h(\alpha))}{|F(\alpha,\lambda_l^h(\alpha))|^2}.\]Since $F_2(\alpha,\lambda_l(\alpha))=O(|\alpha-\alpha_0|^3)$ and $\tilde{c}_{ll}(\alpha,\lambda_l(\alpha))=O(|\alpha-\alpha_0|)$ for $l=1,2$, we have  \[\lim_{\alpha\to\alpha_0}\dfrac{\frac{\partial F_2}{\partial\lambda}(\alpha,\lambda_l^h(\alpha))}{\tilde{c}_{ll}(\alpha,\lambda_l^h(\alpha))}=0.
    \]
Using the above limit together with Lemma \ref{Lem-limits} (c), (d) and $\eqref{222}$, the proof follows.  \end{proof}

\section{threshold eigenvalue case}

\begin{definition}[Model for the threshold case] Let $\alpha_0\in\RR$ be fixed. For $\alpha\in\RR$, consider the operator \begin{equation}\label{model23}
H_\alpha=H_0+\alpha V\quad\text{ on }L^2(\RR^3)
\end{equation} and $u_j,\,j=1,2$ satisfy the following conditions:\begin{enumerate}[label=(\alph*)]
    \item $u_1,u_2\in\mathcal E$ and $u_1\perp u_2$.
    \item $\delta_{jk}+\alpha_0\gamma(\eta_{jk},0)=0$.
    \item $\langle \phi_1, \phi_2 \rangle \neq 0$ or $\|\phi_1\| \neq \|\phi_2\|$ where $\phi_j=\Phi_{u_j,0}$, $j=1,2.$ 
    \item The function $F_2(\alpha,\lambda_j(\alpha))\neq 0$ for $j=1,2$ and $\alpha>\alpha_0$ in $I$.
\end{enumerate}\end{definition}
For this model, by Corollary~\ref{Cor-Modeleigenvalue}, $0$ is a threshold eigenvalue of $H_{\alpha_0}$ with multiplicity two. Furthermore, by Theorem~\ref{Threshold}, when $\alpha>\alpha_0$, this eigenvalue dissolves into the absolutely continuous spectrum, whereas for $\alpha<\alpha_0$ it splits into two simple discrete eigenvalues. 

All the results of Sections~\ref{6}--\ref{7} and Theorem \ref{31} on the behaviour of the average time delay remain valid for this model as $\alpha \to \alpha_0^{+}$, without any modification of the proofs.
\appendix
\begin{appendices}
\section{}
In this Appendix, we state the classical Morse lemma, which gives the local canonical form of a smooth function in a neighborhood of a non-degenerate critical point; see \cite{Milnor} for details.
\bthm[Morse Lemma]
Let $f:\mathbb{R}^2 \to \mathbb{R}$ be a smooth function and let $p=(x_0,y_0)\in\mathbb{R}^2$ be a non-degenerate critical point of $f$, i.e $Df(p)=0$ such that $f(p)=0$ and the Hessian $D^2f(p)$ is non-singular. Then there exist open neighborhoods $U,V \subset \mathbb{R}^2$ of $p$ and $0$, respectively and a diffeomorphism 
\[
\Phi:U \to V, \qquad \Phi(p)=0
\]
such that, in the new coordinates $(\xi,\eta)=\Phi(x,y)$, one has
\[
f\circ\Phi^{-1}(\xi,\eta)=-\xi^2 - \eta^2, \quad f\circ\Phi^{-1}=-\xi^2 + \eta^2, \quad \text{or} \quad f\circ\Phi^{-1}=\xi^2 + \eta^2
\]
according as the Hessian $D^2f(p)$ has index (number of negative eigenvalues) $2$, $1$, or $0$, respectively.
\ethm
\bcor\label{Cor-Morse} Let $f:\mathbb{R}^2 \to \mathbb{R}$ be a smooth function and let $p=(x_0,y_0)$ be a non-degenerate critical point of index $1$ with $f(p)=0$ and $\frac{\partial^2f}{\partial y^2}(p)\neq0$. 
Then there exists a neighborhood $I$ of $x_0$ and distinct smooth functions $y_1,y_2:I\to\mathbb{R}$ such that
\[
f(x,y_j(x)) = 0, \qquad j=1,2
.\]\ecor
\begin{proof}
    By the Morse lemma, $f = gh$ on $U$  where 
    \[g(x,y)=\eta(x,y)+\xi(x,y),\,h(x,y)=\eta(x,y)-\xi(x,y).\] Noting that $g(p)=h(p)=0$, we have
    \[
        0\neq\frac{\partial^2 f}{\partial y^2}(p)
        = 2\,\frac{\partial g}{\partial y}(p)\frac{\partial h}{\partial y}(p),
    \]
    which implies that 
    \[
        \frac{\partial g}{\partial y}(p)\neq0 
        \quad\text{and}\quad 
        \frac{\partial h}{\partial y}(p)\neq0.
    \]
    By the implicit function theorem, there exist a neighbourhood $I$ of $x_0$ and smooth functions $y_1,y_2:I\to\RR$ such that 
    \[
        g(x,y_1(x))=0
        \quad\text{and}\quad
        h(x,y_2(x))=0
        \quad\text{for all }x\in I.
    \]
    This completes the proof.
\end{proof} \section{}
 The following proposition is useful in deducing the spectral concentration. 
 \bpro\label{Pro-specconc}Let $\{T_n\}_{n\in\mathbb{N}}$ be a sequence of self-adjoint operators on a Hilbert
space $\mathcal{H}$ that converges to a self-adjoint operator $T$ in the
\emph{strong resolvent sense} as $n \to \infty$.  
Let $E_n$ and $E$ denote the spectral families of $T_n$ and $T$ respectively.
Let $\{\lambda_n\}_{n\in\mathbb{N}} \subset \mathbb{R}$ be a sequence such that
$\lambda_n \to \lambda $ and set
$P := E(\{\lambda\}).$
Then \[\slim_{n\to\infty}
E_n(-\infty,\lambda_n](I-P)
=
E(-\infty,\lambda).\]

\epro
\begin{proof}
     For each $n\in\mathbb{N}$, define the self-adjoint operator $T_n':=T_n-(\lambda_n-\lambda)$. If $E_n'$ denote the spectral family associated with $T_n'$, then $E_n'(-\infty,\lambda]=E_n(-\infty,\lambda_n].$ Furthermore, it follows that
$T_n' \to T$ in the strong resolvent sense.
Therefore, by Theorem 1.15 in \cite[Chapter VIII]{KatoBook}, we obtain
\[
\slim_{n\to\infty} E'_{n}(-\infty,\lambda](I-P)
=
E(-\infty,\lambda)\] and the result follows.
\end{proof}
\section{}In this appendix, we give the proofs for Lemma~\ref{5.2} and Lemma~\ref{Lem-aijrelation}. \begin{center}
\textbf{Proof of Lemma~\ref{5.2}}
\end{center}
  The proof of  $\lambda_1'(\alpha_0)\geq\frac{1}{\alpha_0^2||\phi_j||^2}$ for $j=1,2$ is trivial. To prove other bound, observe that\begin{equation}\label{iff}
\begin{split}
 \lambda_2'(\alpha_0)\leq\frac{1}{\alpha_0^2||\phi_2||^2}&\iff ||\phi_2||^2(||\phi_1||^2+||\phi_2||^2)-2(||\phi_1||^2||\phi_2||^2-|\langle\phi_1,\phi_2\rangle|^2)\leq||\phi_2||^2d\\&\iff (||\phi_2||^4-||\phi_1||^2||\phi_2||^2+2|\langle\phi_1,\phi_2\rangle|^2)^2\leq||\phi_2||^4d^2\\&\iff 4\langle\phi_1,\phi_2\rangle^4-4||\phi_1||^2||\phi_2||^2|\langle\phi_1,\phi_2\rangle|^2\leq0\\&\iff 4|\langle\phi_1,\phi_2\rangle|^2\left(|\langle\phi_1,\phi_2\rangle|^2-||\phi_1||^2||\phi_2||^2\right)\leq0.
   \end{split}\end{equation} The last inequality is true by the Cauchy-Schwarz inequality. Hence \[  \lambda_2'(\alpha_0)\leq\frac{1}{\alpha_0^2||\phi_2||^2}\leq \frac{1}{\alpha_0^2||\phi_1||^2}.\]Now note that \eqref{iff} is also true if we replace $``\leq"$ by $``="$ which implies \[ \lambda_2'(\alpha_0)=\frac{1}{\alpha_0^2||\phi_2||^2}\iff \phi_1\perp\phi_2 .\] Similarly we can prove \[ \lambda_1'(\alpha_0)=\frac{1}{\alpha_0^2||\phi_1||^2}\iff \phi_1\perp\phi_2 \] which completes the proof.
\qed \vspace{2mm}\begin{center}
\textbf{Proof of Lemma~\ref{Lem-aijrelation}}
\end{center}To prove part (a), note that the phase factor of $ a_{j1} a_{j2}$ and $(-1)^j\langle\phi_1,\phi_2\rangle$ are same. Now\begin{align*}
&| a_{j1} a_{j2}| = \frac{\alpha_0^2 |\langle \phi_1, \phi_2 \rangle| \lambda_j'(\alpha_0)}{d} \\
\iff &
\sqrt{(-1)^{j+1}(-1 + \alpha_0^2 \| \phi_2 \|^2 \lambda_j'(\alpha_0))} \,
\sqrt{(-1)^{j+1}(-1 + \alpha_0^2 \| \phi_1 \|^2 \lambda_j'(\alpha_0))}
= \alpha_0^2 | \langle \phi_1, \phi_2 \rangle | \lambda_j'(\alpha_0) \\
\iff &\quad
1 + \alpha_0^4 \| \phi_1 \|^2 \| \phi_2 \|^2 [\lambda_j'(\alpha_0)]^2
- \alpha_0^2 \lambda_j'(\alpha_0) (\| \phi_1 \|^2 + \| \phi_2 \|^2)
= \alpha_0^4 |\langle \phi_1, \phi_2 \rangle|^2 [\lambda_j'(\alpha_0)]^2.
 \end{align*}Upon rearranging the terms and using \eqref{eqn-lambdaderivative} and 
\eqref{eqn-1111}, we see that the last identity is equivalent to 
\eqref{lambdapath}, which proves (a). For $j=1,2$ we have,
\begin{align*}\frac{a_{j1}^2 + |a_{j2}|^2}{\alpha_0^2\lambda_j'(\alpha_0)} &=\frac{(-1)^{j+1}(-2+\alpha_0^2(||\phi_1||^2+||\phi_2||^2)\lambda_j'(\alpha_0))}{\alpha_0^2\lambda_j'(\alpha_0)d}
.\end{align*} Substituting the values of \(\lambda'_{j}(\alpha_{0})\) from \eqref{eqn-derivativelambda} 
into the above expression on right, we obtain that it equals \(1\).  
This proves part~(b) for the case \(j = k\). Now,
    \begin{align*}&a_{11}a_{21}+a_{12}a_{22}\\&=\frac{1}{a_{11}\overline{a_{22}}}(a_{11}^2\overline{a_{22}}a_{21}+a_{11}a_{12}|a_{22}|^2)\\&=\frac{1}{a_{11}\overline{a_{22}}d^2}\left((-1+\alpha_0^2||\phi_2||^2\lambda_1'(\alpha_0))\alpha_0^2\langle\phi_1,\phi_2\rangle\lambda_2'(\alpha_0)-\alpha_0^2 \langle \phi_1, \phi_2 \rangle \lambda_1'(\alpha_0)(1-\alpha_0^2||\phi_1||_2^2\lambda_2'(\alpha_0))\right)\\&=\frac{\alpha_0^2 \langle \phi_1, \phi_2 \rangle }{a_{11}\overline{a_{22}}d^2}\left(\alpha_0^2\lambda_1'(\alpha_0)\lambda_2'(\alpha_0)(||\phi_1||^2+||\phi_2||^2)-(\lambda_1'(\alpha_0)+\lambda_2'(\alpha_0))\right)\\&=0\end{align*} which proves part (b) for $j=2,k=1$. The proof for the case $j=1,k=2$ is similar.
By part (a) for $j=1,$ we get
   \begin{align*}
||\psi_1||^2&=a_{11}^2 \| \phi_1 \|^2 + |a_{12}|^2 \| \phi_2 \|^2 + 2\Re[a_{11}\overline{a_{12}}\langle\phi_1,\phi_2\rangle]\\
&= \frac{\| \phi_1 \|^2 \left( -1 + \alpha_0^2 \| \phi_2 \|^2 \lambda_1'(\alpha_0) \right)+\| \phi_2 \|^2 \left( -1 + \alpha_0^2 \| \phi_1 \|^2 \lambda_1'(\alpha_0) \right)- 2 \alpha_0^2| \langle \phi_1, \phi_2 \rangle|^2 \lambda_1'(\alpha_0)}{d}
 \\&=\frac{- \left( \| \phi_1 \|^2 + \| \phi_2 \|^2 \right)
+ 2 \alpha_0^2 \lambda_1'(\alpha_0) \left( \| \phi_1 \|^2 \| \phi_2 \|^2
- |\langle \phi_1, \phi_2 \rangle|^2 \right)}{d}\\&=\frac{- \left( \| \phi_1 \|^2 + \| \phi_2 \|^2 \right)+\left( \| \phi_1 \|^2 + \| \phi_2 \|^2+d \right)}{d}=1
\end{align*} 
Similarly, we get \[||\psi_2||^2=|a_{21}|^2||\phi_1||^2+|a_{22}|^2||\phi_2||^2+2 \Re[a_{21}\overline{a_{22}}\langle\phi_1,\phi_2\rangle]=1\] and
 \[\langle\psi_1,\psi_2\rangle=a_{11} a_{21}||\phi_1||^2+a_{12}\overline{a_{22}}||\phi_2||^2+ (a_{11}\overline{a_{22}}\langle\phi_1,\phi_2\rangle+a_{12} a_{21}\langle\phi_2,\phi_1\rangle)=0.\] This completes the proof of part (c).
\qed
\end{appendices}
   
 \section*{Acknowledgements}
 The authors gratefully acknowledge Professor K.B. Sinha for numerous discussions regarding this work on resonance phenomena as well as educating us on many mathematical aspects of quantum mechanics. 
 
The first author acknowledges the support received from NBHM (under DAE, Govt. of
India) Ph.D. fellowship grant 0203/7/2019/RD-II/14855.

\bibliographystyle{alpha}
 
\end{document}